\documentclass[11pt, reqno]{amsart}
\usepackage{amsmath}
\usepackage{amsthm}
\usepackage{amssymb}
\usepackage{amscd,mathrsfs}
\usepackage{amsbsy}
\usepackage{color}
\usepackage{epsfig}
\usepackage{graphicx}
\usepackage{psfrag}
\usepackage[normalem]{ulem}
\usepackage{bbm}
\usepackage{hyperref}
\usepackage{changebar}
\usepackage{comment}
\usepackage[T1]{fontenc}
\providecommand{\norm}[1]{\lVert#1\rVert}

\usepackage[marginpar=.8 in]{geometry}

\usepackage{tikz}
\usetikzlibrary{arrows}
\usetikzlibrary{decorations.pathreplacing}

\usepackage[T1]{fontenc} 
\usepackage[UKenglish]{babel}
\usepackage[UKenglish]{isodate}

\theoremstyle{plain}
\newtheorem{theorem}{Theorem}[section]
\newtheorem{remark}[theorem]{Remark}
\newtheorem{proposition}[theorem]{Proposition}
\newtheorem{lemma}[theorem]{Lemma}
\newtheorem{corollary}[theorem]{Corollary}

\newtheorem{definition}[theorem]{Definition}
\newtheorem*{claim}{Claim}

\newtheorem*{lemma*}{Lemma}

\theoremstyle{definition}

\newtheorem{example}[theorem]{Example}

\makeatletter
\newcommand{\labeltext}[2]{%
  \@bsphack
  \csname phantomsection\endcsname 
  \def\@currentlabel{#1}{\label{#2}}%
  \@esphack
}
\makeatother

\newfont\bbf{msbm10 at 12pt}

\def\R{{\mathbb R}}

\def\N{{\mathbb N}}
\def\Z{{\mathcal Z}}

\def\B{{\mathcal B}}

\def\P{{\mathcal P}}

\def\H{{\mathcal H}}
\def\A{{\mathcal A}}
\def\U{{\mathcal U}}

\def\V{{\mathcal V}}
\def\L{{\mathcal L}}
\def\Q{{\mathcal Q}}

\def\Y{{\mathcal Y}}

\def\es{{\emptyset}}

\def\1{\ensuremath{{\mathbbm{1}}}}
\def\Tau{{\mathcal T}}

\def\supp{\mbox{\rm supp}}

\def\diam{\mbox{\rm diam} }

\def\le{\leqslant}
\def\ge{\geqslant}

\newcommand{\I}{\mathcal{I}}

\newcommand{\hU}{\mathring{U}}

\newcommand{\Lp}{\mathcal{L}}

\newcommand{\tmu}{\tilde{\mu}}

\newcommand{\hLp}{\mathring{\Lp}}

\newcommand{\hf}{\mathring{f}}

\newcommand{\E}{\mathbb{E}}

\newcommand{\beq}{\begin{equation}}
\newcommand{\eeq}{\end{equation}}

\renewcommand{\i}{\mathbf{i}}
\renewcommand{\j}{\mathbf{j}}
\renewcommand{\k}{\mathbf{k}}
\newcommand{\md}{\dim_\textup{M}}
\newcommand{\hd}{\dim_\textup{H}}

\DeclareMathOperator*{\essinf}{ess\,inf}

\newcommand{\vertiii}[1]{{\left\vert\kern-0.25ex\left\vert\kern-0.25ex\left\vert #1 
    \right\vert\kern-0.25ex\right\vert\kern-0.25ex\right\vert}}
\newcommand{\invertiii}[1]{{\vert\kern-0.25ex\vert\kern-0.25ex\vert #1 
    \vert\kern-0.25ex\vert\kern-0.25ex\vert}}

\numberwithin{equation}{section}

\begin{document}

\title{Cover times in dynamical systems}
\date{\today}


\author[N.~Jurga]{Natalia Jurga}
\address{Mathematical Institute\\
University of St Andrews\\
North Haugh\\
St Andrews\\
KY16 9SS\\
Scotland} 
\email{naj1@st-andrews.ac.uk}
\urladdr{https://nataliajurga.wordpress.com/}

\author[M.~Todd]{Mike Todd}
\address{Mathematical Institute\\
University of St Andrews\\
North Haugh\\
St Andrews\\
KY16 9SS\\
Scotland} 
\email{m.todd@st-andrews.ac.uk}
\urladdr{https://mtoddm.github.io/}

\thanks{NJ was supported by an EPSRC Standard Grant (EP/R015104/1) and Leverhulme Early Career Fellowship (ECF-2021-385). MT was partially supported by the FCT (Funda\c{c}ão para a Ciência e a Tecnologia) 
project 2022.07167.PTDC. The authors thank the referee for useful comments, which significantly improved the paper.}

\begin{abstract}
We introduce the notion of cover time to dynamical systems.  This quantifies the rate at which orbits become dense in the state space and can be viewed as a global, rather than the more standard local, notion of recurrence for a system. Using transfer operator tools for systems with holes and inducing techniques, we obtain an asymptotic formula for the expected cover time in terms of the decay rate of the measure of the ball of minimum measure.   We apply this to a wide class of uniformly hyperbolic and non-uniformly hyperbolic interval maps, including the Gauss map and Manneville-Pomeau maps.
\end{abstract}

\maketitle

\section{Introduction and summary of results}

Let $X \subseteq Y$ be a subset of a compact metric space $(Y,d)$ and consider a map $f: X \to X$ equipped with an ergodic probability measure $\mu$. In this article we introduce and study the \emph{cover time} of the system $(f,\mu)$ which, roughly speaking, quantifies the rate at which orbits become dense in $X$. More precisely, given $\delta>0$ we define the $\delta$-covering time $\tau_\delta: X \to \N\cup \{\infty\}$,
\begin{equation}
\tau_\delta(x):= \inf\left\{n \ge 1\;:\; \{x, f(x), \ldots, f^n(x)\}\; \textnormal{is $\delta$-dense in $X$}\right\}
\end{equation}
where we say that a subset $Z' \subset Z$ is $\delta$-dense in $Z$ if for all $x \in Z$ there exists $x' \in Z'$ such that $d(x,x')\le \delta$.
We will be interested in the asymptotic behaviour of the expected cover time with respect to $\mu$, $\E_\mu (\tau_\delta)$ as $\delta \to 0$.

In probability, analogous notions to the cover time have a rich history of study, such as in the setting of random walks on graphs \cite{aldous-graph,lovasz,dlp} and Markov chains \cite{peres}, as well as being extensively studied  in stochastic geometry \cite{flatto, janson,penrose, aldous}. Moreover, applications are numerous and include wireless communications \cite{wireless} and genomics \cite{genomics}.  However the cover time, despite being a very natural dynamical notion to study in connection with other better understood notions of recurrence such as hitting and return times, has not yet been addressed in the field of dynamical systems. We briefly describe an application which illustrates the utility of obtaining estimates on the cover time. Suppose $f:\Lambda \to \Lambda$ is an expanding repeller. Then the expected cover time $N=\E_\mu(\tau_\delta)$  describes how long one would should expect to wait for an orbit segment $\{x,f(x), \ldots, f^N(x)\}$ to produce a $\delta$-approximation of the repeller, which would be useful for an efficient computation of its image. 

It is an interesting open programme to investigate what types of asymptotic behaviour are feasible for the expected cover time and to characterise these in terms of the basic properties of the dynamical system. We point to similar lines of research in the setting of random walks on graphs, such as investigation into general lower and upper bounds on the expected cover time in terms of the number of vertices in the graph and a conjecture that the smallest expected cover time should be attained by complete graphs \cite{lovasz}. In this paper we make fundamental progress towards this general goal by obtaining an asymptotic formula for the expected cover time for a wide class of uniformly and non-uniformly hyperbolic systems. 

 Define
\begin{equation}\label{minmeas}
M_\mu(\delta)=\min_{x \in \mathrm{supp}(\mu)} \mu(B(x,\delta))
\end{equation}
where $B(x,\delta)$ denotes an open ball, noting that the minimum exists by compactness of $\mathrm{supp}(\mu)$ and lower semicontinuity of $x \mapsto \mu(B(x,\delta))$. For a wide class of one-dimensional systems we will show that $\E_\mu(\tau_\delta)$ scales roughly like $M_\mu(\delta)^{-1}$ as $\delta \to 0$. In particular, if we let $\md \mu$ denote the Minkowski dimension of $\mu$
\begin{equation}
\md \mu= \lim_{\delta \to 0} \frac{\log M_\delta(\mu)}{\log \delta},\label{md}
\end{equation}
which was introduced in \cite{ffk}\footnote{This differs from the definition for the Minkowski dimension which was given in \cite{ffk}, but the proof that \eqref{md} is an equivalent characterisation can be found in \cite[Lemma 1.1]{bjk}. Note that the limit in \eqref{md} will exist for the measures considered in this paper, however this is not the general case. We also note that this definition does not agree with the definition of the Minkowski dimension given by Pesin \cite[\S 7]{pesin} and that some authors would refer to \ref{md} as the $L^{-\infty}$ dimension (see e.g. \cite[Proposition 4.2]{ffk}).} then our main result says that if $\md \mu<\infty$ then $\E_\mu(\tau_\delta)$ scales roughly like $\delta^{-\md \mu}$ as $\delta \to 0$. This is in correspondence with the asymptotic behaviour of the expected hitting time $\E_\mu(\tau_{B(x,\delta)})$ (see \eqref{hitting}) to a shrinking ball centred at some $x \in X$ which grows like $O(\delta^{-D_\mu(x)})$ as $\delta \to 0$ where $D_\mu(x)$ denotes the local dimension of $\mu$ at $x$. In other words, while hitting times are governed by \emph{local} scaling properties of the measure $\mu$, cover times are governed by the scaling properties of the (globally determined) ball of minimum measure. Our setup also gives rise to natural examples of systems which cover the state space at a rate slower than $\delta^{-\alpha}$ (for any $\alpha>0$), namely where the expected cover time $\E_\mu(\tau_\delta)$ has an exponential dependence on $1/\delta$, see \S \ref{examples}. We obtain more precise bounds on $\E_\mu(\tau_\delta)$, which can be found in Theorems~\ref{cover} and \ref{cover2} in the uniformly hyperbolic case and in Theorems~\ref{cover-induced} and \ref{cover-induced2} in the non-uniformly hyperbolic case. We remark that the subexponential terms in the asymptotic expression for the expected cover time are interesting since they capture how nuanced information about the system plays a role in the expected cover time, such as how many balls of radius $\delta$ and measure $\approx M_\mu(\delta)$ there are.

In \cite{JurMor20, bjk} the cover time of the chaos game was studied, which is a random algorithm that was introduced by Barnsley \cite[Chapter 3]{BarnsleyBook} for constructing the attractors of iterated function systems. Assuming the iterated function system satisfies an appropriate separation condition, this is analogous to studying $\E_\mu(\tau_\delta)$ for a system $(f,\mu)$, where $f$ is necessarily a full and finite branched uniformly expanding map (constructed using the inverses of the contractions belonging to the iterated function system) and $\mu$ is the stationary measure associated to the chaos game. In the current paper we study the cover time from a dynamical viewpoint and in particular we study systems which are well beyond the scope of \cite{bjk} such as systems which are non-uniformly hyperbolic (subexponentially mixing), are not full branched, or have infinitely many branches. In particular, our results give partial answers to \cite[Questions 6.3 and 6.4]{bjk}.

The cover time $\tau_\delta$ is closely related to the hitting time $\tau_U :X \to \N\cup \{\infty\}$
\begin{equation} \label{hitting}
\tau_U(x)=\inf\{n \ge 1: f^n(x) \in U\}
\end{equation}
 for $U=B(x,\delta)$ where $x \in X$. For background on hitting time statistics in dynamical systems see for example \cite{Sau09}, \cite[Chapter 5]{EVbook}. In the uniformly hyperbolic case, a key tool in this paper will be to adapt the tools based on the spectral theory of transfer operators for dynamical systems with a family of holes shrinking to a point \cite{BruDemTod18}, to instead give bounds on the expected hitting time $\E_\mu(\tau_U)$ over a family $\U_\delta$ of holes $U=B(x,\delta)$ centred at $x \in X$ and which cover $X$.  We have to do this in a uniform way, which is a significant technical challenge, particularly when dealing with the `short returns' which requires a new symbolic argument given in \S\ref{section:return}, which restricts the geometry of our setup (see \eqref{eq:ratios} below). Following this, we will adapt an approach of Matthews \cite{matthews}, based on randomising the order in which balls $U$ are visited, which will allow us to express the expected cover time in terms of $\max_{U} \E_\mu (\tau_U)$, where the maximum is taken over balls $U=B(x,\delta)$ with $x \in \Lambda$. Finally, given a non-uniformly hyperbolic system $(f,\mu)$ we will consider a suitable first return map $F=f^{\tau_Y}$ which is uniformly hyperbolic, and show that the cover time for the original system $(f,\mu)$ can be estimated in terms of the cover time of the induced system $(F, \frac{1}{\mu(Y)}\mu|_Y)$.

\textbf{Notation.} We say $x=O(y)$ if there exists $C>0$ such that $|x|\le C|y|$.  We will also allow the constants $C,c>0$ to take different values throughout the paper: they will indicate that the bounds are uniform in the context in which they are used.

{\bf Organisation.} In \S \ref{section:uniform}  we state the main theorems in the uniformly expanding case (Theorems \ref{cover} and \ref{cover2}) and study the transfer operators which will be used to obtain estimates on the expected hitting times in this setting. In \S \ref{section:return} we obtain estimates on expected return times, which are then used in \S \ref{section:hit} to obtain estimates on expected hitting times. In \S \ref{section:matthews} we prove Theorems \ref{cover} and \ref{cover2} by obtaining a dynamical generalisation of Matthews' approach \cite{matthews} to express the expected cover time in terms of the expected hitting times. In \S \ref{section:nonuniform} we state the main theorems in the non-uniformly expanding case (Theorems \ref{cover-induced} and \ref{cover-induced2}) and use an inducing argument to prove it.  In \S~\ref{examples} we give applications of our results, both in the uniformly hyperbolic and non-uniformly hyperbolic cases. This includes applications to Gibbs-Markov maps, the Gauss map and Manneville-Pomeau systems as well as examples of slowly covering systems.  We postpone the proofs of Propositions~\ref{prop:u} and \ref{prop:u'} to the appendix: these guarantee that for all $\delta$ sufficiently small, we can find a family of balls of radius roughly $\delta$, each of whose expected hitting time can be studied by our transfer operator approach.

\section{Setup and results for the exponentially mixing case}\label{section:uniform}

In this section we introduce a class of maps and potentials which lead to the system being exponentially mixing, and state the main results in this setting. Furthermore we introduce the transfer operators which will be used to estimate hitting times. We establish the uniform spectral properties for these operators which will be required to obtain uniform estimates on expected hitting times.  From here on, all of our dynamical systems will be interval maps, which we formalise next.

Let $I \subset \R$ be a bounded interval and $\mathcal{Z}=\{Z_i\}_{i\in \I}$ be a finite or countable collection of subintervals of $I$ with disjoint interiors.  Let $f:\bigcup_{n \in \I}Z_n\to I$ be continuous and strictly monotonic on each $Z_n$. We denote $f_i:=f|_{Z_i}$.  We call an interval $H\subset I$ on which $f$ is not defined a \emph{hole}\footnote{These initially refer to the holes present in the original dynamics, but later this idea will be also used when we will investigate hitting times via an open dynamical systems perspective, where an extra hole is introduced.}, for example when there is a gap between $Z_i$s. Write $D=I\setminus \bigcup_i Z_i$.  We denote 
$$\Lambda:=\left\{x\in I: f^k(x) \in \bigcup_{n \in \I} Z_n \text{ for all } k\ge 0\right\}$$
to be the repeller of $f$ and study the dynamical system $f:\Lambda \to \Lambda$. In our notation we may sometimes suppress the fact that we are restricting the dynamics to $\Lambda$, for instance by writing $f(U)$ for some $U$ which is not a subset of $\Lambda$ (for example $U$ may be an interval, while $\Lambda$ may be a topological Cantor set); this is implicitly understood as $f(U)=f(\Lambda \cap U)$. In many cases we consider, $I= \bigcup_{n \in \I} Z_n$ (in particular there are no holes), in which case $\Lambda=I$. Let $\mathcal{Z}^n$ be the set \emph{$n$-cylinders}, i.e., of maximal intervals $Z$ of $I$ such that $f^k(Z)\subset Z_{i_k}$ for some $Z_{i_k}\in \mathcal{Z}$ for $k=0, \ldots, n-1$.

We now outline our basic assumptions on our map $f$ and invariant measure $\mu$.

\subsection{Basic assumptions on $f$ and $\mu$} \label{props}

We assume there exists a nonatomic Borel probability measure $m$ such that $m(D)=0$ and $m$ is conformal\footnote{From a thermodynamic formalism point of view, our assumptions imply that the pressure of $\phi$ is zero.} with respect to a potential $\phi:\Lambda \to \R$, i.e. $\frac{dm}{d(m\circ f)}=e^\phi$.  We set $\phi|_D=-\infty$ and write $S_n\phi=\sum_{i=0}^{n-1}\phi\circ f^i$. Let 
\begin{equation} \label{op}
\L \psi(x):= \sum_{y \in f^{-1}x} \psi(y)e^{\phi(y)}
\end{equation}
for $\psi \in L^1(\Lambda,m)$, (we will actually consider $\Lp$ acting on functions of bounded variation, see Section~\ref{ssec:trans_closed}). We assume $\phi$ satisfies the following regularity properties as in \cite[Section 2]{BruDemTod18}: 
\begin{enumerate}
\item [(a1)]\labeltext{(a1)}{a1} $\exists C_d>0$ such that $|e^{S_n\phi(x)-S_n\phi(y)}-1| \le C_d|f^n(x)-f^n(y)|$ whenever $f^i(x), f^i(y)$ lie in the same element of $\Z$ for each $i=0,1, \ldots, n-1$;
\item  [(a2)]\labeltext{(a2)}{a2}   $\exists n_0 \in \N$ such that $\sup_\Lambda e^{S_{n_0} \phi}<\inf_{\Lambda\setminus D} \L^{n_0} 1$;
\item  [(a3)]\labeltext{(a3)}{a3}  For each $x \in \Lambda$ and $\delta>0$ such that $J=B(x,\delta)$ has the property that $J\cap D=\emptyset$, $\exists N=N(J)$ such that $\inf_{\Lambda \setminus D} \L^N \1_{J\cap \Lambda}>0$.
\end{enumerate}
Note that \ref{a3} implies $\supp(m)=\Lambda$.  
By conformality of $m$, $\int \Lp^n1dm=\int 1dm=1$ so that $\inf_{\Lambda\setminus D} \Lp^n1 \le 1$ for all $n \in \N$. Hence by \ref{a2}, $\sup_\Lambda e^{S_{n_0}\phi}<1$. By this and conformality of $m$:
\begin{equation} \text{there exist } C>0, \omega>0 \text{ such that any $n$-cylinder has measure } \le Ce^{-\omega n}.
\label{eq:cyls}
\end{equation}
 Since $\sup_\Lambda e^{S_{n_0}\phi}<1$ and since $\sup_\Lambda e^{S_n\phi}$ is submultiplicative, we have that
\begin{equation} \label{finite-sum}
\sum_{n=1}^\infty \sup_\Lambda e^{S_n\phi}< \infty.
\end{equation}
By the same reasoning
\begin{equation}\label{n1}
\exists n_1 \in \N\; \textnormal{such that} \; (2n_1+5)(C_d+1)\sup_\Lambda e^{S_{n_1}\phi}<1.
\end{equation}

We will also assume that
\begin{enumerate}
\item [(a4)]\labeltext{(a4)}{a4}  $\exists c_m>0$ such that $\inf_{Z \in \Z^{n_1}} m(f^{n_1}(Z)) \ge c_m$.
\end{enumerate}
Along with the cylinder structure, this implies that for all $1\le i \le n_1$,
\begin{equation}
\inf_{Z \in \Z^i} m(f^i(Z)) \ge \inf_{Z \in \Z^{n_1}} m(f^{n_1}(Z)) \ge c_m.
\label{a4i}
\end{equation}
By \ref{a1}, $\sup_{Z} e^\phi \le (1+C_d)m(Z)/m(f(Z))$ so applying \eqref{a4i} with $i=1$ yields
\begin{equation} \label{eq:a2'}
\sum_{Z \in \Z} \sup_Z e^\phi \le (1+C_d) \sum_{Z \in \Z} \frac{m(Z)}{m(f(Z))} \le (1+C_d)c_m^{-1}<\infty.
\end{equation}

As described in \cite[Section 2.1]{BruDemTod18}, under \ref{a1}--\ref{a4} (which as shown above imply (F1)--(F4) in that paper) we can apply \cite[Theorem 1]{rychlik} to show that $\Lp$ admits a unique invariant measure $\mu$ which is absolutely continuous with respect to $m$ and whose density $g$ is bounded away from 0 and of bounded variation. Moreover the system is exponentially mixing.  By this and by conformality of $m$, \eqref{finite-sum} implies
\begin{equation} \label{K}
K:=1+\sum_{n=1}^\infty \sup_{Z \in \Z^n} \mu(Z)<\infty.
\end{equation}

It will sometimes be more convenient to work with the symbolic coding of $f$. Recall that the index set $\I$ labels the intervals of monotonicity $\{Z_i\}_{i\in \I}$ for $f$. Let $\Sigma \subset \I^\N$ be the subshift given by the set of sequences $\i=(i_0, i_1, \ldots) \in \I^{\N_0}$  for which there exists an $x \in \Lambda$ such that $f^{n}(x) \in Z_{i_n}$ for all $n \in \N_0$. Define the projection $\Pi:\Sigma \to I$ as
$$\Pi(i_0, i_1,  \cdots) :=\lim_{n \to \infty} f_{i_0}^{-1} \circ \cdots \circ f_{i_{n-1}}^{-1}(I),$$
i.e. $\Pi(\i)=x$. Then $f\circ \Pi=\Pi \circ \sigma$ where $\sigma: \Sigma \to \Sigma$ is the left shift map. We let $\Sigma_n$ denote the set of all words of length $n$ in $\Sigma$ and $\Sigma^*$ denote the set of all finite words in $\Sigma$. Given $w \in \Sigma^*$ we write $[w]:=\{\i \in \Sigma: \i|n=w\}$ where $\i|n=(i_0, \ldots, i_{n-1})$ if $\i=(i_n)_{n \in \N_0}$ (sometimes for brevity we will write $w= i_0\ldots i_{n-1}$). We write $\tmu$ to be the measure on $\Sigma$ such that $\mu=\Pi_*\tmu$.

We will require two additional assumptions on $\tmu$:
\begin{enumerate}
\item [(a5)]\labeltext{(a5)}{a5}  $\tmu$ is  \emph{quasi-Bernoulli}; i.e. there exists $C_*>1$ such that for all finite words $\i,\j \in \Sigma^*$, $C_*^{-1}\tmu([\i])\tmu([\j]) \le \tilde\mu([\i\j]) \le C_* \tmu([\i])\tmu([\j])$;
\item[(a6)]\labeltext{(a6)}{a6}  $\tmu$ is \emph{$\psi$-mixing}, i.e. $$\left|\frac{\tmu\left(\cup [x_0, \ldots, x_{i-1}, y_0, \ldots, y_{j-1}, z_0, \ldots, z_{\ell-1}]\right)}{\tmu\left([x_0, \ldots, x_{i-1}]\right) \tmu\left([z_0, \ldots, z_{\ell-1}]\right)} -1\right| \le \gamma(j)$$
where the union is taken over all words $y_0 \ldots y_{j-1}$ of length $j$ such that \\$(x_0, \ldots, x_{i-1}, y_0, \ldots, y_{j-1}, z_0, \ldots, z_{\ell-1})\in \Sigma_{i+j+\ell}$, and $\gamma(j)\to 0$ as $ j\to \infty$.
\end{enumerate}

\subsection{Gibbs-Markov maps}\label{GM}
In the case where $f:\Lambda \to \Lambda$ is Markov, i.e. for each $Z\in \Z$,  $f(Z)$ is a union of elements of $\Z$, then the class of maps satisfying \ref{a1}-\ref{a6} is precisely  the class of Gibbs-Markov maps with the big images and pre-images (BIP) property. These maps and their properties are discussed in more detail in \cite[Section 2.6.2]{BruDemTod18}, but we give a brief account here. 

We say that $f$ satisfies the \emph{big images and pre-images (BIP) property} if there exists a finite set
$\{ Z_j \}_{j \in \mathcal{J}} \subset \Z$ such that 
$\forall Z \in \Z$, $\exists j, k \in \mathcal{J}$ such that
$f(Z_j) \supseteq Z$ and $f(Z) \supseteq Z_k$.
We also assume that $|Df| \ge \gamma^{-1} > 1$ on each $Z \in \Z$. We say that an invariant measure $\mu$ (or $\tilde{\mu}$) is \emph{Gibbs} if there exists a Lipschitz continuous potential $\phi$ (i.e., with some uniform Lipschitz constants such that on each $Z\in \Z$, $\phi$ is Lipschitz) and constants $K, P$ such that for all $i_0, \ldots, i_{n-1} \in \Sigma_n$ and $\i \in [i_0, \ldots, i_{n-1}]$,
$$K^{-1}\le \frac{\tilde{\mu}([i_0,\ldots ,i_{n-1}])}{e^{S_n\phi(\Pi(\i))+nP}}\le K.$$

We say that $(f,\mu)$ is Gibbs-Markov with BIP if $f$ is Markov and satisfies BIP and $\mu$ is Gibbs. One can readily check that such maps satisfy \ref{a1}-\ref{a6}. Indeed, as in \cite[Section 2.6.2]{BruDemTod18}, these maps satisfy \ref{a1}-\ref{a3}. Also \ref{a4} follows from the big images property and Markov property; \ref{a5} follows because Gibbs measures $\tilde{\mu}$ are quasi-Bernoulli. Finally \ref{a6} holds since Gibbs measures $\tilde{\mu}$ are $\psi$-mixing \cite{Bra83}. 

On the other hand, suppose $f$ is Markov and satisfies \ref{a1}-\ref{a4}. Then $\mu$ is necessarily a Gibbs measure. Moreover by \cite{Sar03}, $f$ must satisfy BIP. In other words, the class of Markov maps satisfying \ref{a1}-\ref{a6} coincides with the class of Gibbs Markov map with BIP.  We call $f_{i}=f|_{Z_i}$ the \emph{branch} of $f$ at $Z_i$; if $f(Z_i) = I$ then we say that $f$ has a \emph{full branch} on $Z_i$.  The simplest BIP examples are when all branches of $f$ are full.

\subsection{Other assumptions on $f$ and $m$.}
Before we state our main theorems, we will need to introduce further assumptions on the system.

Firstly, we assume that for some constant $C_m<\infty$ and $s>0$,
\begin{equation} \label{measdiam}
m(U) \le C_m \diam(U)^s \hspace{0.5cm} \textnormal{for all intervals $U \subset I$} .
\end{equation}

We will also require the derivatives of iterates of $f$ to satisfy bounded distortion: there exists a constant $C_{bd}>0$ such that for all $x, y \in \Lambda$ for which $f^i(x),f^i(y)$ are in the same cylinder $Z$ for each $0 \le i \le n-1$ then
\begin{equation}
C_{bd}^{-1}\le \frac{|(f^i)'(x)|}{|(f^i)'(y)|} \le C_{bd}.
\label{bd}
\end{equation}
$f$ satisfies \eqref{bd} for example when $f$ is uniformly expanding and has finitely many $C^{1+\alpha}$ branches. Similarly, when $f$ is uniformly expanding, has infinitely many $C^{1+\alpha}$ branches and there is a uniform upper bound bound on the H\"older constants associated to the first order derivative of each branch, $f$ also satisfies \eqref{bd}. 

Next, in order that $f:I \to I$ can be sufficiently well approximated by the symbolic representation of the system, we require adjacent cylinders to be comparable.  That is there is $C>0$ such that given $Z, Z'\in \Z^n$ where $Z, Z'\subset Z_{n-1}\in \Z^{n-1}$ which are \emph{adjacent}, meaning that there is an interval $A\subset Z_{n-1}$ such that $Z$ and $Z'$ are the only elements of $\Z^n$ intersecting $A$, 
\begin{equation}
\frac1C\le \frac{\mu(Z)}{\mu(Z')} \le C.
\label{eq:ratios}
\end{equation}
Note that here we are assuming that given $Z_{n-1}\in \Z^{n-1}$, each $Z\in \Z^n$ in $Z_{n-1}$ has an adjacent $Z'\in \Z^n$ in $Z_{n-1}$, which implies that if domains of $\Z^n$  accumulate in $Z_{n-1}\in \Z^{n-1}$, then this must occur at the boundary of $Z_{n-1}$.  Observe that it is sufficient to check this assumption for elements of $\Z^1$, where $I$ is the corresponding element of $\Z^0$.

Finally, we will sometimes assume that $m$ is Ahlfors regular: namely that there exists $c>0$ such that for any $x \in \Lambda$ and $r>0$,
\begin{equation}
\frac{r^{s_f}}{c} \le m(B(x,r)) \le cr^{s_f}
\label{eq:m_scale}
\end{equation}
where $s_f=\hd \Lambda$. This holds if $m$ is Lebesgue, or more generally if $m$ is conformal for a potential $-s_f\log|Df|$ and, for example, $f$ is a finite branch Markov map\footnote{To see this, approximate the measure $m(B(x,r))$ from above and below by the measure of appropriate cylinders and apply conformality and \ref{a1} (or alternatively using assumption \eqref{bd}).}, in which case this is referred to as the \emph{Hausdorff measure}.

\subsection{Results}

We begin by stating our main theorem in the special case that $m$ satisfies \eqref{eq:m_scale}, e.g. if $f$ is Markov and $m$ is conformal for a potential $-s_f\log|Df|$.

\begin{theorem}\label{cover}
Assume $(f,\mu)$ satisfies \ref{a1}-\ref{a6} and \eqref{measdiam}--\eqref{eq:m_scale}.  There exist $0<c<C<\infty$
such that for all $\delta>0$,
$$c\delta^{-s_f} \le \E_\mu(\tau_\delta) \le C\delta^{-s_f}\log(1/\delta).$$
Moreover, if the system is Gibbs Markov and $f$ has at least 2 full branches then we also have the sharp lower bound:
$$c\delta^{-s_f}\log(1/\delta) \le \E_\mu(\tau_\delta) \le C\delta^{-s_f}\log(1/\delta).$$
\end{theorem}

The assumption that $f$ has at least 2 full branches allows us to find, for all $\delta>0$ sufficiently small, sufficiently many $\delta$-balls centred at a point in $\Lambda$ which have measure close to $M_\mu(\delta)$ and are `far' from each other in the sense that there is a good lower bound on the time in between consecutive visits to these balls. This is enough to provide a sharp lower bound, however we expect that there are weaker assumptions which would yield the same conclusion.

In the case where there is a sharp lower bound, one may wonder whether the limit
$$\lim_{\delta \to 0} \frac{\E(\tau_\delta)}{\delta^{-s_f} \log(1/\delta)}$$
exists. In the probability theory literature on cover times of random walks, the existence of the analogous limit has been established for some specific examples, such as for the expected cover time of a disk by a random walk in $\mathbb{Z}^2$ \cite{coverdisk} and the expected cover time of the binary tree by a simple random walk \cite{covertree}.

We also consider more general measures $\mu$ which may not satisfy \eqref{eq:m_scale}. Suppose $\md \mu<\infty$. For each $\delta>0$ we denote 
\begin{equation}\label{er}
\mathrm{Err}(\delta):= \left|\md \mu- \frac{\log \left(M_\mu(\delta))\right)}{\log \delta}\right|.
\end{equation}
Clearly $\md \mu<\infty$ implies $\lim_{\delta \to 0}\mathrm{Err}(\delta)=0$. Therefore in the case that $\md \mu<\infty$, one can think of $\delta^{-\md \mu \pm \mathrm{Err}(\delta)}$ as being upper and lower bounds on the measure of the ball of minimum measure at scale $\delta$.

\begin{theorem}\label{cover2}
Assume $(f,\mu)$ satisfies  \ref{a1}-\ref{a6} and \eqref{measdiam}--\eqref{eq:ratios}. There exist $0<c<C<\infty$ and $\varepsilon>0$ such that such that for all $\delta>0$,
\begin{equation} \label{eq:inf}
\frac{c}{M_\mu(\delta/\varepsilon)} \le \E_\mu(\tau_\delta) \le \frac{C}{M_\mu(\varepsilon\delta)}\log(1/\delta).
\end{equation}
In particular if $\md \mu<\infty$ then
\begin{equation}\label{eq:fin}
c\delta^{-\md \mu+\mathrm{Err}(\delta/\varepsilon)} \le \E_\mu(\tau_\delta) \le C\delta^{-\md \mu -\mathrm{Err}(\varepsilon \delta)}\log(1/\delta).
\end{equation}
\end{theorem}

Note that $\varepsilon=\frac{t}{2T}$ where $t,T$ are given by (U)(a) below. We note that if the measure $\mu$ is doubling (for example, if $f$ is a finite branch Markov map and $\mu$ is Gibbs), then the dependence on $\varepsilon$ can be removed from the estimates. For example in \eqref{eq:inf} the denominators appearing in the bounds could be replaced by $M_\mu(\delta)$ (the constants coming from the doubling property would be absorbed in the constants $c$ and $C$). The upper bound in \eqref{eq:fin} is analogous to the upper bound in Theorem \ref{cover}. The reason that $\mathrm{Err}(\delta)$ does not appear in Theorem \ref{cover} is because for a general measure $\mu$, the measure of the ball of minimum measure at scale $\delta$ may take some time to resemble the asymptotic limit $O(\delta^{\md \mu})$, whereas if  \eqref{eq:m_scale} holds then this can already be seen at large scales $\delta$.

In order to obtain a sharper lower bound in Theorem \ref{cover2}, roughly speaking we would require an estimate on the number of balls of measure $O(\delta^{\md \mu+\mathrm{Err}(\delta/\varepsilon)})$ seen at scale $\delta$. This is straightforward in the case  \eqref{eq:m_scale} holds, since all balls of comparable diameter have comparable measure, which ensures a sharp lower bound in Theorem \ref{cover}. However we note that it is not always true that we have exponentially many balls of measure $O(\delta^{\md \mu+\mathrm{Err}(\delta/\varepsilon)})$ at scale $\delta$ (meaning the system can cover faster than $\log(1/\delta)$ times the reciprocal of the measure of the ball of minimum measure at scale $\delta$) see for instance \cite[Theorem 1.1(2)]{JurMor20}.  In the other hand, in \S~\ref{examples} we provide two examples of slowly covering systems: systems for which $\md \mu= \infty$ and provide estimates on the asymptotic growth of their expected cover times.

\subsection{Uniformly large images for family of punctured maps}

 Suppose $(f,\mu)$ satisfies \ref{a1}-\ref{a4}, in particular so that the constants $n_1$ from \eqref{n1} and $c_m$ from \ref{a4} are well-defined. The main result of this section is Lemma \ref{bigimages} where we show a uniform large images property for a family of punctured maps, where each punctured map is obtained from $(f,\mu)$ by introducing a `hole' $U$ in the system. We say that the system $(f,\mu)$ satisfies \labeltext{(U)}{U} (U) if  there exists $\delta_0>0$ such that for each $0<\delta\le \delta_0$ we can find a finite collection $\U_\delta$ of closed subintervals of $U \subset I$ which satisfy the following assumptions: 

\begin{enumerate}
\item [{(U)(a)}]\labeltext{(U)(a)}{Ua}
There exists $0<t<1<T$ such that for all  $U \in\U(\delta_0):=\bigcup_{0<\delta\le \delta_0} \bigcup \{U: U \in \U_\delta\}$, there exists $x \in \Lambda$ such that $B(x,t\delta) \subseteq U \subseteq  B(x,T\delta)$;

\item [{(U)(b)}]\labeltext{(U)(b)}{Ub} for any $0<\delta\le \delta_0$ the interiors of the intervals in $\U_\delta$ are pairwise disjoint;

\item [{(U)(c)}]\labeltext{(U)(c)}{Uc}  for each $0<\delta \le \delta_0$, $\Lambda \subseteq \bigcup_{U \in \U_\delta}  U$;

\item [{(U)(d)}]\labeltext{(U)(d)}{Ud} there exists $0<\tilde{\beta} <c_m$ such that for all $1 \le i \le n_1$, $Z \in \Z^i$ and $U \in \U(\delta_0)$ either $Z \subseteq U$ or $\frac{m(U \cap Z)}{m(Z)} \le \frac{\tilde{\beta}}{(n_1+1)(1+C_d)^2}$, where $n_1$ comes from \eqref{n1};
\item [{(U)(e)}]\labeltext{(U)(e)}{Ue}  for any $U \in \U_\delta$, $U \cap \Lambda=\Pi(\bigcup_{\i \in \P_\delta}[\i])$ where $\P_\delta \subset \Sigma^*$ is a finite or countable collection of words $\i$ with the property that $|\i| =O(\log(1/\delta))$ (where the implied constant is uniform over all $0<\delta \le \delta_0$).
\end{enumerate}

Under the additional assumptions \eqref{measdiam} and \eqref{bd}, $(f,\mu)$ satisfies \ref{U}. Note that the only place where \eqref{measdiam} and \eqref{bd} will be required is in the proofs of the following two propositions, which are postponed to the appendix.

\begin{proposition}\label{prop:u}
Suppose $(f,\mu)$ satisfies \ref{a1}-\ref{a4},  \eqref{measdiam} and \eqref{bd}. Then $(f,\mu)$ satisfies \ref{U}. 
\end{proposition}

To obtain the sharp lower bound in Theorem \ref{cover} (in the case  \eqref{eq:m_scale} holds) we will also require the following proposition.

\begin{proposition}\label{prop:u'}
Suppose $(f,\mu)$ is Gibbs-Markov with at least two full branches, and satisfies  \eqref{measdiam} and \eqref{bd}. There exists $\delta_0>0$ such that for each $0<\delta\le \delta_0$  we can find a finite collection 
$$\V_\delta \subset \{\Pi([wab^{n_3}])\;:\; w \in \{a,b\}^*\}$$
of pairwise disjoint subsets of $\Lambda$ where: 
\begin{enumerate}
\item [(a)] $a,b \in \Sigma_1$, $a \neq b$ are such that $f_a$ and $f_b$ are full branched and $n_3 \in \N$ is chosen such that $C_*^3\tilde{\mu}([ab^{n_3}])K<1$;
\item [(b)] for all $0\le i \le n_3$ and $U,V \in \V_\delta$, $f^i(U)\cap V=\emptyset$;
\item [(c)] there exists $c, \varepsilon>0$ such that $\#\V_\delta \ge c\delta^{-\varepsilon}$;
\item[(d)]  there exists $T>t>0$ such that for all $U \in \V_\delta$, there exists $x \in \Lambda$ such that $B(x,t\delta)\cap \Lambda \subset U \subset  B(x,T\delta)$;
\item [(e)] there exists $0<\tilde{\beta} <c_m$ such that for all $1 \le i \le n_1$, $Z \in \Z^i$ and $U \in \U(\delta_0)$ either $Z\cap \Lambda \subseteq U$ or $\frac{m(U \cap Z)}{m(Z)} \le \frac{\tilde{\beta}}{(n_1+1)(1+C_d)^2}$, where $n_1$ comes from \eqref{n1};
\item [(f)] for any $U=\Pi([wab^{n_3}]) \in \V_\delta$, $|w| =O(\log(1/\delta))$ (where the implied constant is uniform over all $0<\delta \le \delta_0$);
\item [(g)] if  $\Pi([w_1ab^{n_3}]), \Pi([w_2ab^{n_3}]) \in \V_\delta$ are distinct, then $w_1ab^{n_3}$ is not a subword of $w_2ab^{n_3}$.
\end{enumerate}
\end{proposition}

Proposition \ref{prop:u'} essentially states that for sufficiently small $\delta$ we can find sufficiently many balls of diameter roughly $\delta$ and measure roughly $M_\mu(\delta)$ (this is (c) and (d)), which are dynamically far from each other (namely there is a uniform lower bound on the time taken in between visits of any orbit to any two balls in $\V_\delta$ - this is (b)). Proposition \ref{prop:u'} will only be used in the proof of the lower bound in Theorem \ref{cover}.

The proof of Theorems \ref{cover} and \ref{cover2} will require us to study a family of punctured dynamical systems which have a hole at some $U \in \U(\delta_0)$. Namely, given $U \in \U(\delta_0)$ we define the (punctured) map with a hole at $U$ by $\hf_U= f|_{\Lambda \setminus U}$. Then its iterates are given by $\hf_U^n=f|_{\hU^{n-1}}$ where $\hU^{n-1}=\bigcap_{i=0}^n f^{-i}(\Lambda\setminus U)$ (so that $\hU^0=\Lambda\setminus U$). We will require that the transfer operators associated to these holes have a uniform spectral gap, and this will follow from \ref{Ud} via the following lemma.

\begin{lemma} \label{bigimages}
Assume $f$ and $\mu$ satisfy \ref{a1}-\ref{a4}, \eqref{measdiam} and \eqref{bd}. Then there exists $c_0>0$ such that for all $U \in \U(\delta_0)$,
$$\inf_{U \in \U(\delta_0)} \inf \{m(\hf_U^{n_1} J): J \in \Z^{n_1} \; \textnormal{s.t.} \; J \cap \hU^{n_1} \neq \emptyset\} \ge c_0$$
where $n_1$ comes from \eqref{n1}. 
\end{lemma}

\begin{proof}  Throughout the proof we'll assume $\Lambda=I$, the more general case follows similarly after taking intersections with $\Lambda$. Fix $Z \in \Z^{n_1}$, $Z=\Pi([i_0, \ldots, i_{n_1-1}])$ and $U \in \U(\delta_0)$ such that $Z\cap \hU^{n_1} \neq \emptyset$. We want to show there exists $c_0$ which is independent of $U \in \U(\delta_0)$ such that $m(\hf_U^{n_1}(Z)) \ge c_0$. 

For each $0 \le j \le n_1$ either $f^{-j}(U) \cap Z=\emptyset$ or $U \cap \Pi([i_{j}, \ldots, i_{n_1}] )\neq \emptyset$ and $f^{-j}(U) \cap Z=f_{i_0\ldots i_{j-1}}^{-1}(U \cap\Pi( [i_{j}, \ldots, i_{n_1-1}]))$, where $f_{i_0\ldots i_{j-1}}^{-1}=f_{i_0}^{-1}\circ \cdots f_{i_{j-1}}^{-1}$ corresponds to the inverse branch of $f^j$ which maps $I$ homeomorphically to $\Pi([i_0\ldots i_{j-1}])$ (in particular if $j=0$ we define $f_{i_0\ldots i_{j-1}}^{-1} =\mathrm{Id}$). Note that this means that we cannot have $\Pi([i_{j} \ldots i_{n_1-1}] )\subset U$ as this would imply that $f^{-j}(U) \cap Z=Z$ which is a contradiction since we've assumed $Z \cap \hU^{n_1} \neq \emptyset$.

Hence 
\begin{align*}\frac{m(f^{-j}(U) \cap Z)}{m(Z)} & \le \frac{m(U \cap \Pi([i_{j}\ldots i_{n_1-1}]))}{m([\Pi(i_{j}\ldots i_{n_1-1}]))} \frac{\sup_Z e^{S_j \phi}}{\inf_Z e^{S_j \phi}} \\
&\le (1+C_d)\frac{m(U \cap \Pi([i_{j}\ldots i_{n_1-1}]))}{m(\Pi([i_{j}\ldots i_{n_1-1}]))} \le \frac{\tilde{\beta}}{(n_1+1)(1+C_d)}
\end{align*}
where the first inequality follows by conformality, the second by \ref{a1} and the third by Proposition \ref{prop:u} and  \ref{Ud}.
Thus
\begin{align}\frac{m(Z \setminus \hU^{n_1})}{m(Z)} &= \frac{m(Z \cap\bigcup_{j=0}^{n_1}f^{-j}(U))}{m(Z)} \le (n_1+1)\frac{\tilde{\beta}}{(n_1+1)(1+C_d)}=\frac{\tilde{\beta}}{1+C_d}.\label{proportion2}\end{align}
In particular
\begin{align*}
m(\hf_U^{n_1}(Z))=m(f^{n_1}(Z))-m(f^{n_1}(Z \setminus \hU^{n_1}))&\ge m(f^{n_1}(Z))-\frac{m(Z \setminus \hU^{n_1})}{m(Z)}\frac{\sup_Z e^{S_{n_1}\phi}}{\inf_Z e^{S_{n_1}\phi}}\\
&\ge m(f^{n_1}(Z))-\frac{m(Z \setminus \hU^{n_1})}{m(Z)}(1+C_d) \ge c_m-\tilde{\beta}>0
\end{align*}
where we have used conformality in the first inequality, \ref{a1} in the second inequality and \eqref{proportion2} in the third inequality. The proof follows by setting $c_0:=c_m-\tilde{\beta}$.  
\end{proof}

\subsection{Transfer operators for closed system.} 
\label{ssec:trans_closed} We will study the action of the operator $\Lp$ on the Banach space $\B$ of functions $\psi: \Lambda \to \R$ with bounded variation norm $\norm{\cdot}=\norm{\cdot}_1+|\cdot|_{BV}$, where $\norm{\cdot}_1$ denotes the $L^1$ norm with respect to $m$ and $|\cdot|_{BV}:=|\cdot|_{BV,I}$ where for any interval $J \subseteq I$,
$$|\psi|_{BV,J}:=\sup\left\{\sum_{i=0}^{k-1} |\psi(x_{i+1})-\psi(x_i)| \; : \; x_0<x_1<\cdots<x_k, \; x_i \in J\cap \Lambda, \; \forall i\le k\right\}$$
where the supremum is taken over all finite sets $\{x_i\}_i \subset J\cap \Lambda$.

Under \ref{a1}-\ref{a4} we have the following.

\begin{lemma}\label{prels}
Assuming \ref{a1}-\ref{a4}, for all $n \ge 0$,
\begin{enumerate}
\item [(a)] $\sum_{Z \in \Z^n} e^{S_n\phi}<\infty$;
\item [(b)] for each $Z \in \Z$, $|\phi|_{BV,Z} \le C_d\sup_Z e^\phi$;
\item [(c)] $|e^{S_n\phi}|_{BV}<\infty$.
\end{enumerate}
\end{lemma}

\begin{proof}
(a) follows by induction on \eqref{eq:a2'}. (b) follows from (a1) since $|e^{\phi(x_{i+1})}-e^{\phi(x_i)}|\le C_d e^{\phi(x_i)}$ for any $\{x_i\}_{i=1}^k \subset Z$. For $n=0$, (c) follows from (b) and \eqref{eq:a2'}. Setting $\phi|_D=-\infty$ only adds a term bounded by the series in \eqref{eq:a2'} to the variation. For $n \ge 1$ the claim follows by induction.
\end{proof}

Using Lemma \ref{prels} and \eqref{n1}, the operator $\Lp^{n_1}$ satisfies the assumptions of \cite[Theorem 1]{rychlik}, thus $\Lp^{n_1}$ is quasi-compact on $\B$. By \cite[Theorem 3]{rychlik} and \ref{a3}, $\Lp^{n_1}$ has a simple eigenvalue at 1 and no other eigenvalues of modulus 1, in other words $\Lp^{n_1}$ has spectral gap. By Lemma \ref{prels}, since $\Lp$ is a bounded operator on $\B$, it also has a spectral gap. We let $g \in \B$ denote the normalised ($\int gdm=1$) eigenfunction of $\Lp$ associated to the leading eigenvalue 1. $g$ is bounded away from 0. Write $d\mu=gdm$.

\subsection{Transfer operators for open systems.}
Here we use spectral properties of appropriate transfer operators which will later give us information on hitting time statistics.  This type of approach has been used previously, for example in \cite{Dem05}, \cite{kl2} and \cite{BruDemTod18}, but we have to make some adaptations, principally to Proposition~\ref{ly} below in order to get uniform estimates for our setting.  Lemma~\ref{pert} and Corollary~\ref{decomp} are then essentially standard, but we include the short proofs for completeness.

\begin{definition}[Transfer operators for open systems] For any hole $U$ let $\hLp_{U}: \B \to \B$ denote the transfer operator with a hole at $U$, given by
$$\hLp_{U}(\psi):=\sum_{y \in \hf_U^{-1}x} \psi(y)e^{\phi(y)}= \Lp(\1_{I \setminus U}  \psi).$$
\end{definition}

Note that the iterates of $\hLp_{U}$ are given by
\begin{equation} \label{iterates}
\hLp_U^n \psi(x)=\sum_{y \in \hf_U^{-n}x} \psi(y)e^{S_n\phi(y)}=\Lp^n(\psi \1_{\hU^{n-1}}).
\end{equation}
By conformality of $m$,
\begin{equation}\label{integral}
\int \hLp_U^n \psi dm=\int \Lp^n(\psi \1_{\hU^{n-1}})dm= \int_{\hU^{n-1}} \psi dm.
\end{equation}

By (U) we have the following set of uniform Lasota-Yorke inequalities for the family of operators $\{\hLp_U\}_{U \in \U(\delta_0)}$.

\begin{proposition}\label{ly}
There exist $C_0>0$ and $0<\alpha<1$ such that for any $\varphi \in \B$, $U \in \U(\delta_0)$ and $n \ge 0$,
\begin{equation*}
\norm{\hLp_{U}^n\varphi} \le C_0\alpha^n\norm{\varphi} +C_0 \int_{\hU^{n-1}}|\varphi|dm.
\end{equation*}
The above is also true when $\hLp_U$ is replaced by $\Lp$ and $\hU^{n-1}$ is replaced by $\Lambda$.
\end{proposition}

\begin{proof} Throughout the proof we'll assume $\Lambda=I$, the proof in the more general case is similar. We begin by noting that our holes $U$ are closed, meaning that each $\hU^n$ is relatively open in $I$, which is in contrast to much of the literature where the hole is taken to be open and $\hU^n$ is closed. For each $J \in \Z^{n}$ there is a unique choice of disjoint intervals $\{(x_1,x_2), \ldots, (x_{k_J-1},x_{k_J})\}$ such that
$$J \cap \hU^{n-1}=\bigcup_{i=1}^{k_J/2} (x_{2i-1},x_{2i}).$$
In particular note that $x_1$ is the left hand end point of $J$ and $x_{k_J}$ is the right hand end point of $J$.

Given $\psi \in \B$,
$$
|\hLp_U^{n}\psi|_{BV} \le \sum_{J \in \Z^n} |\psi e^{S_n\phi}|_{BV,J}+\psi(x_1)e^{S_n\phi(x_1)}+ \cdots +\psi(x_{k_J})e^{S_n\phi(x_{k_J})}.$$
Note that for all $x \in \{x_1, \ldots, x_{k_J}\}$,
$$\psi(x)e^{S_n\phi(x)} \le \frac{1}{m(J\cap \hU^{n-1})} \int_{m(J \cap \hU^{n-1})} |\psi|e^{S_n\phi}dm+ |\psi e^{S_n\phi}|_{BV,J}$$
yielding
\begin{align}|\hLp_U^n \psi|_{BV} \le \sum_{J \in \Z^n} (k_J+1) |\psi e^{S_n\phi}|_{BV,J} +\frac{k_J}{m(J \cap \hU^{n-1})} \int_{J \cap \hU^{n-1}} |\psi|e^{S_n\phi}dm. \label{ly1}
\end{align}
Next, note that
\begin{align}
|\psi e^{S_n \phi}|_{BV,J} &\le \sup_J \psi |e^{S_n\phi}|_{BV,J} +\sup_J e^{S_n \phi} |\psi|_{BV,J} \nonumber\\
& \le \sup_J \psi  C_d \sup_J e^{S_n\phi}+\sup_J e^{S_n\phi} |\psi|_{BV,J} \label{a1used2} \\
& \le \sup_J e^{S_n \phi} \left(|\psi|_{BV,J}(C_d+1) +\frac{C_d}{m(J\cap \hU^{n-1})} \int_{J \cap \hU^{n-1}} |\psi|dm \right) \nonumber
\end{align}
where in the second line we have used Lemma \ref{prels}(b) and in the last line we have bounded $\sup_J \psi \le \frac{1}{m(J\cap \hU^{n-1})} \int_{J \cap \hU^{n-1}} |\psi|dm + |\psi|_{BV,J}.$

Combining with \eqref{ly1} we obtain
\begin{align}&|\hLp_U^n \psi|_{BV} \nonumber \\
& \le \sum_{J \in \Z^n} (k_J+1) \sup_J e^{S_n\phi} |\psi|_{BV,J}(C_d+1) +((k_J+1) \sup_J e^{S_n\phi} C_d+k_J) \frac{1}{m(J \cap \hU^{n-1})} \int_{J \cap \hU^{n-1}} |\psi|e^{S_n\phi}dm. \label{ly2}
\end{align}
By conformality and \ref{a1}, for all $x \in J \cap \hU^{n-1}$,
\begin{equation}
e^{S_n\phi(x)} \frac{m(f^n(J \cap \hU^{n-1}))}{m(J\cap \hU^{n-1})} \le 1+C_d. \label{a1used3}
\end{equation}
Also, note that $k_J \le 2n+4$, since $J \cap \hU^{n-1}$ contains at most $n+1$ holes corresponding to $U, \ldots, f^{-n}(U)$. Combining these with \eqref{ly2} we obtain
\begin{align}
|\hLp_U^n\psi|_{BV} &\le (2n+5) \sup_I e^{S_n\phi}(C_d+1) |\psi|_{BV} \nonumber\\
&\;\;\;+\sum_{J \in \Z^n} ((2n+5)\sup_I e^{S_n\phi} C_d+2n+4)(1+C_d) \frac{1}{m(f^{n}(J \cap \hU^{n-1}))}\int_{J \cap \hU^{n-1}} |\psi|dm \\
&\le  (2n+5) \sup_I e^{S_n\phi}(C_d+1) |\psi|_{BV}+((2n+5)\sup_I e^{S_n\phi} C_d+2n+4)(1+C_d) \frac{1}{c_0}\int_{ \hU^{n-1}} |\psi|dm \label{ly3}
\end{align}
where we have used Lemma \ref{bigimages} in the final line. 

Applying \eqref{ly3} with $n=n_1$ we see that for $\tilde{\alpha}= (2n_1+5) \sup_I e^{S_{n_1}\phi}(C_d+1) |<1$ we have
$$|\hLp_U^{n_1}\psi|_{BV} \le \tilde{\alpha}|\psi|_{BV}+ ((2n_1+5)\sup_I e^{S_{n_1}\phi} C_d+2n_1+4)(1+C_d) \frac{1}{c_0}\int_{ \hU^{n_1-1}} |\psi|dm .$$
We can then iterate this to complete the proof with $\alpha=\tilde{\alpha}^{1/n_1}$. 
\end{proof}

By Proposition \ref{ly}, the compactness of the unit ball of $\B$ in $L^1(m)$ and the conformality of $m$ that the spectral radius of $\hLp_{U}$ acting on $\B$ is at most 1 and its essential spectral radius is bounded by $\alpha<1$, for each  $U \in \U(\delta_0)$ and the same is true for $\Lp$. Thus $\Lp$ and each $\hLp_{U}$ are quasi-compact as operators on $\B$. Define the following perturbative norm for operators $P,Q$ acting on $\B$:
$$\invertiii{P-Q}:=\sup\{|P\varphi -Q \varphi|_1 \;:\; \norm{\varphi}\le 1\}.$$

\begin{lemma}\label{pert} For any $U \subset I$, 
$$\invertiii{\Lp-\hLp_{U}} \le m( U) \le C_1\mu( U)$$
where $C_1=1/\essinf g.$
\end{lemma}

\begin{proof} If $\varphi \in \B$ with $\norm{\varphi} \le 1$ then
$$\int |(\Lp-\hLp_U)\varphi|dm= \int|\Lp(1_U\varphi)|dm \le \sup_\Lambda |\varphi| m(U)\le m(U)$$
by conformality. The fact that $\essinf g>0$ follows from \ref{a3}.
 \end{proof}

\begin{corollary}\label{decomp} There exists $\delta_1 \in (0,\delta_0)$ such that for all  $U \in \U(\delta_1)$, $\hLp_{U}$ and $\Lp$ have uniform spectral gap. In particular, $\Lp$ and $\hLp_{U}$ admits the following spectral decomposition
\begin{align*}
\Lp&=\Pi+R\\
\hLp_{U}&=\lambda_{U}\Pi_{U}+R_{U}
\end{align*}
where 
\begin{enumerate}
\item [(a)] $\lambda_U$ is the leading eigenvalue of $\hLp_U$ for the normalised ($\int g_{U}dm=1$) eigenfunction $g_U$, i.e. $\hLp_{U}g_{U}=\lambda_{U}g_{U}$   (analogously the leading eigenvalue of $\Lp$ is 1, with normalised eigenfunction $g$); \item [(b)] $\Pi_{U}$ is the projection onto the eigenspace spanned by $g_U$ (analogously $\Pi$ is the projection onto the eigenspace spanned by $g$), moreover $\hLp_U^*m_U=\lambda_Um_U$ for the eigenmeasure $m_U$; 
\item [(c)] for $\alpha<\beta_1<\beta_2<1$ we have  $\lambda_{U}>\beta_2$ and the spectral radius $\sigma(R_{U})<\beta_1$ (also $\sigma(R)<\beta_1$); 
\item [(d)] $\Pi_{U}^2=\Pi_{U}$ and $\Pi_{U}R_{U}=R_{U}\Pi_{U}=0$ (analogously $\Pi^2=\Pi$ and $\Pi R=R\Pi=0$).
\end{enumerate}

Moreover, there exist constants $K_1, K_2, K_3,K_4>0$ and $\eta \in (0,1)$ such that for all  $U \in \U(\delta_1)$,
\begin{align}
\invertiii{\Pi_{U}-\Pi}& \le K_1 m(U)^\eta \label{k1}\\
\norm{R_{U}^n}& \le K_2 \beta_1^n \label{k2}\\
|1-\lambda_U| &\le K_3\mu(U)\label{ev-bound}\\
\norm{g_U}& \le K_4 \label{k4}
\end{align}\label{kl}
\end{corollary}

\begin{proof} As already discussed in the previous section, $\Lp$ has a spectral gap by \cite{rychlik}. Using Proposition \ref{ly} and Lemma \ref{pert} and \cite[Corollary 1]{kl} there exists $\delta_1 \in (0,\delta_0)$ such that the family $\{\hLp_U\}_{U \in \U(\delta_1)}$ has a uniform spectral gap and the stated eigendecomposition including properties (a)-(d) (note that although \cite{kl} studies just one sequence of perturbed operators, we can obtain analogues of the bounds in \cite[Corollary 1]{kl} (and thus the uniform spectral gap) over our whole family of operators $\{\hLp_U\}_{U \in \U(\delta_1)}$ due to the uniformity of the constants in Proposition \ref{ly}).   \eqref{k1} follows from \cite[Corollary 1(1)]{kl}, \eqref{k2} follows from \cite[Corollary 2(2)]{kl} and \eqref{k4} follows from \cite[Corollary 1(2)]{kl} (again the constants are uniform due to  the uniformity of the constants in Proposition \ref{ly}.

To see \eqref{ev-bound} note that
\begin{align*}
|1-\lambda_U|&= \left| \int g_Udm-\lambda_U \int g_U dm\right| =  \left| \int \Lp g_Udm- \int \hLp_U g_U dm\right| \\
&=  \left| \int (\Lp-\hLp_U)g_Udm\right| \le \sup_{U\in \U(\delta_1)} \norm{g_U} \invertiii{\Lp-\hLp_U} \le C_1  K_4\mu(U)
\end{align*}
where the last inequality follows by Lemma \ref{pert} and \eqref{k4}. 
\end{proof}

\section{Return time estimates} \label{section:return}

We will estimate the expected hitting times to $U \in \U_\delta$ in terms of the top eigenvalue $\lambda_U$. In order to then obtain Theorem \ref{cover} we will need to estimate $\lambda_U$ in terms of the measure of $U$ (which will be done in Proposition \ref{eig-bound}). To this end, we will first require the following estimates on return times to $U$, which is the main focus of this section (note that the following proposition, and in particular in the lemma following it, is the only place we use \eqref{eq:ratios}). We denote $\mu_U=\frac{1}{\mu(U)}\mu|_U$.

\begin{proposition} \label{prop:RTS}
Assume $(f,\mu)$ satisfies \ref{a1}-\ref{a6} and \eqref{measdiam}--\eqref{eq:m_scale}.  Then there exist $c>0$ and $\delta_2 \in (0,\delta_1)$ such that for all $U \in \U(\delta_2)$ and $n\ge 1/\mu(U)$,
$$\mu_U(\tau_U\ge n)\ge c\lambda_U^{n}.$$
\end{proposition}

We will require the following lemma, which allows us to approximate intervals by a pair of well-understood sets of cylinders.  Let $d(Z)$ be the depth of a cylinder $Z$.  Note that given a cylinder $[x_0, \ldots, x_k]$, for $\ell\le k$, its \emph{$\ell$-prefix} is $(x_0, \ldots, x_\ell)$.

\begin{lemma}  Assume \eqref{eq:ratios}.
\begin{enumerate}
\item If $U_I \subset U \subset U_O$ and $n\in \N$ then 
$$\frac{\mu(U_O)}{\mu(U)}\mu_{U_O}(\tau_{U_O}<n) \ge \mu_U(\tau_U<n) \ge \frac{\mu(U_I)}{\mu(U)}\mu_{U_I}(\tau_{U_I}<n).$$
\item There exists $\kappa>0$ with the following property.  For an interval $U$, there are depths $d_L, d_R$ such that there are $U_I\subset U \cap \Lambda \subset U_O$ with $\mu(U_I) \ge (1-\kappa)\mu(U)$ and $\mu(U_O)\le (1+\kappa)\mu(U)$.  Moreover, $U_I$ and $U_O$ are each the $\Pi$ image of (at most) a union of $(d_L+1)$-cylinders $\{U_L^i\}_i$ and $(d_R+1)$-cylinders $\{U_R^i\}_i$ where all the $\{U_L^i\}_i$ have the same $d_L$-prefix and all the $\{U_R^i\}_i$ have the same $d_R$-prefix.
\end{enumerate}
\label{lem:approx}
\end{lemma}

\begin{proof}
For the first part,
$$\mu_U(\tau_U<n) = \frac1{\mu(U)}\mu\left(x\in U: \tau_U(x)<n\right) \ge \frac1{\mu(U)}\mu\left(x\in U_I: \tau_{U_I}(x)<n\right)= \frac{\mu(U_I)}{\mu(U)}\mu_{U_I}\left( \tau_{U_I}<n\right),$$
and similarly for the upper bound.

For the second part, for an interval $U$, let $U_L$ and $U_R$ be the deepest cylinders such that $\Pi(U_L)$ and $\Pi(U_R)$ are adjacent to each other, $\Pi(U_L)$ to the left, so that $U\cap \Lambda\subset \Pi(U_L\cup U_R)$. 

\begin{claim}
There exists $B\in (0, 1)$ such that for $\{U_L^i\}_i$ the depth $d(U_L)+1$ cylinders contained in $U_L\cap \Pi^{-1}(U)$ and  $\{U_R^i\}_i$ the depth $d(U_R)+1$ cylinders contained in $U_R\cap \Pi^{-1}(U)$,
$$\mu\left(\left(\bigcup_iU_L^i\right)\cup\left(\bigcup_iU_R^i\right)\right)\ge B\mu(U).$$
Moreover adding in $U_L^\ell\subset U_L$ and $U_R^r\subset U_R$, the depth $d(U_L)+1$ and depth $d(U_R)+1$ cylinders respectively so that $\Pi^{-1}(U)\subset U_L^\ell\cup U_R^r\cup \left(\bigcup_iU_L^i\right)\cup\left(\bigcup_iU_R^i\right)$, 
$$\mu\left(U_L^\ell\cup U_R^r\cup \left(\bigcup_iU_L^i\right)\cup\left(\bigcup_iU_R^i\right) \right)\le \frac{\mu(U)}B.$$
\end{claim}

\begin{proof}[Proof of Claim.]  This claim follows from the cylinder structure.  It is sufficient to show that  
$$\frac{\mu\left(U_L^\ell\cup \left(\bigcup_iU_L^i\right)\right)}{\mu\left(\left(\bigcup_iU_L^i\right)\right)}$$
is uniformly bounded above (the case for the right-hand cylinder follows similarly).  First note that $\bigcup_iU_L^i\neq \es$ since otherwise $U_L$ is not the deepest cylinder we could have chosen.  Then since $U_L^\ell$ must be adjacent to some $U_L^i$, the bound follows directly from \eqref{eq:ratios}.
\end{proof}
\end{proof}

\begin{proof}[Proof of Proposition~\ref{prop:RTS}.]
The problems here are when there are multiple overlaps, so the worst case is when the symbolic model is the full shift, so we will assume this here.  Moreover we will assume that $\Lambda = I$: the adaptation to the general case goes through by intersecting intervals with $\Lambda$.

We first show what happens for the non-conditional measure case.

\begin{claim}
There exist constants $0<c(U,n)<1<C(U,n)<\infty$ with the property that 
$$c(U,n)\lambda_U^n \le \mu(\tau_U\ge n)\le C(U,n) \lambda_U^n$$ where $c(U,n), C(U,n) \to 1$ as $\mu(U) \to 0$ and $n \to \infty$.
\end{claim}

\begin{proof}
\begin{align*}
\mu(\tau_U\ge n)&=\int  \1_{\{\tau_{U}\ge n\}}\cdot g~dm =  \int \hLp_{U}^n(   g)~dm=\lambda_{U}^n\left( \int \Pi_{U}(g) + \lambda_{U}^{-n}R_{U}^n(g) dm\right) \\
&= \lambda_{U}^n\left( \int cg_{U}+ \lambda_{U}^{-n}R_{U}^n(g) dm\right)=\lambda_{U}^n\left(  c + \lambda_{U}^{-n}\int R_{U}^n(g) dm\right)
\end{align*}
for the constant $c=\int gdm_U$, by definition of $\Pi_U$.  Then
$$|c-1|=\left|\int cg_U~dm-\int g~dm\right| = \left|\int\Pi_U(g)~dm-\int \Pi(g)~dm\right| \le K_1\mu(U)^\eta \|g\|\le K_1K_4\mu(U)^\eta$$
by \eqref{k1} and \eqref{k4}.  Since also from \eqref{k2},
$$\left|\int R_U^n(g)~dm\right|\le \|R_U^n(g)\|\le K_2\beta_1^n\|g\| \le K_2K_4\beta_1^n,$$
we obtain 
$$\mu(\tau_U\ge n) \ge \lambda_U^n\left(1-K_1K_4\mu(U)^\eta-K_2K_4\left(\frac{\beta_1}{\beta_2}\right)^n\right),$$
completing the proof of the lower bound, and the upper bound is similar.
\end{proof}

The idea is to `add buffers' of length $m$ so that we lose some of the power coming from $\lambda_U$, but gain better distortion estimates using $\psi$-mixing \ref{a6}.

\textbf{Single cylinder case}

We start by assuming $\Pi^{-1}(U)$ is a cylinder, and generalise later on. To simplify notation, we switch to using the symbolic measure $\tilde{\mu}$ and assume $U$ itself is a cylinder.

Given $m \in \N$ and assuming $n\ge 2k+2m$, we write
\begin{align*}
\tilde{\mu}_U(\tau_U\ge n) & = \tilde{\mu}_U\left(\underline x=(x_0, x_1, \ldots) \in U: [u_0, \ldots, u_{k-1}]\notin \cup_{i=0}^{n-2}[x_{i+1}, \ldots, x_{i+k}]\right)\\
& = \tilde{\mu}_U\left(\underline x \in U: [u_0, \ldots, u_{k-1}]\notin \cup_{i=k+m}^{n-2}[x_{i+1}, \ldots, x_{i+k}]\right)\\
&\quad- \tilde{\mu}_U\big(\underline x \in U: [u_0, \ldots, u_{k-1}]\in \cup_{i=0}^{k+m-1}[x_{i+1}, \ldots, x_{i+k}] \\
&\hspace{4cm} \text{ \& } [u_0, \ldots, u_{k-1}]\notin \cup_{i=k+m}^{n-2}[x_{i+1}, \ldots, x_{i+k}]\big).
\end{align*}
We refer to these terms as $(I)$ and $(II)$ 

For $(I)$, using $\psi$-mixing \ref{a6},
\begin{align*}
\tilde{\mu}_U\left(\underline x \in U: [u_0, \ldots, u_{k-1}]\notin \cup_{i=k+m}^{n-2}[x_{i+1}, \ldots, x_{i+k}]\right) &\\
&\hspace{-7cm} =  \frac1{\tilde{\mu}(U)}\tilde{\mu}\left(\underline x \in U: [u_0, \ldots, u_{k-1}]\notin \cup_{i=k+m}^{n-2}[x_{i+1}, \ldots, x_{i+k}]\right) \\
&\hspace{-7cm} \ge (1-\gamma(m)) \tilde{\mu}\left(\underline x \in \Sigma: [u_0, \ldots, u_{k-1}]\notin \cup_{i=0}^{n-k-m-2}[x_{i+1}, \ldots, x_{i+k}]\right) \sim \lambda_U^{n-k-m-2},
\end{align*}
where the final step follows from the claim since
\begin{align*}
\mu\left(\underline x \in \Sigma: [u_0, \ldots, u_{k-1}]\notin \cup_{i=0}^{n-k-m-2}[x_{i+1}, \ldots, x_{i+k}]\right)&=\mu\left(\underline x \in \Sigma : \tau_U(x) \ge n-k-m-2\right).
\end{align*}
 Then for $(II)$ we look at the larger value 
 \begin{align*}
\tilde{\mu}_U\left(\underline x \in U: [u_0, \ldots, u_{k-1}]\in \cup_{i=0}^{k+m-1}[x_{i+1}, \ldots, x_{i+k}] \text{ \& } [u_0, \ldots, u_{k-1}]\notin \cup_{i=2k+2m}^{n-2}[x_{i+1}, \ldots, x_{i+k}]\right).
\end{align*}
There are three conditions here: being in $U$ as well as the two conditions on the cylinders.  By $\psi$-mixing \ref{a6} these are almost independent.  The problematic part here, the \emph{short returns}, is dealt with by the following claim.  We note that the arguments here are similar to those in \cite{AbaLam13}, but here we have more general systems (our measures need not be Bernoulli) and ultimately have to deal with more than one cylinder.

\begin{claim}  There is $\tilde\theta<1$ independent of $U$ such that
 \begin{align*}
 \tilde{\mu}_U\left(\underline x \in U: [u_0, \ldots, u_{k-1}]\in \left\{[x_{i+1}, \ldots, x_{i+k}]\right\}_{i=0}^{k+m-1}\right) \le \tilde\theta.
 \end{align*}
 \end{claim}
 
 \begin{proof}[Proof of Claim]
 We first estimate
 \begin{align*}
 \tilde{\mu}_U\left(\underline x \in U: [u_0, \ldots, u_{k-1}]\in \left\{[x_{i+1}, \ldots, x_{i+k}]\right\}_{i=0}^{k-1}\right) \le \tilde\theta,
 \end{align*}
 so we are interested in estimates of the form 
 $$\tilde{\mu}_U\left([u_0, u_1, \ldots, u_i, u_0, \ldots, u_{k-1}]\right) = \frac1{\tilde{\mu}(U)}\tilde{\mu}\left([u_0, \ldots, u_{k-1}] \cap[u_0, u_1, \ldots, u_i, u_0, \ldots, u_{k-1}]\right).$$
We can deal with this case by case: if the intersection here is non-empty and $i=0$ then $u_j=u_0$ for all $1\le j<k$ and then
 \begin{align*}
\tilde{ \mu}_U\left([u_0,u_0,\ldots, u_{k-1}] \right)& = \tilde{\mu}_U\left([u_0,\ldots, u_{k-1}, u_{k-1}]\right)= 1 -  \frac1{\tilde{\mu}(U)}\tilde{\mu}\left(\cup_{u_{k-1}'\neq u_{k-1}}[u_0,\ldots, u_{k-1}, u_{k-1}']\right)\\
&  \le 1-\frac1{C_*} \tilde{\mu}\left(\cup_{u_{k-1}'\neq u_{k-1}} [u_{k-1}']\right) = 1-\frac1{C_*} \left(1-\tilde{\mu}([u_{k-1}])\right)<1
  \end{align*}
  by \ref{a5}.  Crucially here we do not have to consider any further overlaps, since in this case they are accounted for by the first one.
  
  Similarly, if the relevant intersection is non-empty for $i=1$ with $u_0\neq u_1$ but $u_0=u_2$, then in fact $u_i=u_j$ if and only if $i=j \mod 2$.  So 
  \begin{align*}
 \tilde{\mu}_U\left([u_0, u_1,u_0,\ldots, u_{k-1}] \right)& = \tilde{\mu}_U\left([u_0,\ldots, u_{k-1},u_{k-2},  u_{k-1}]\right)\\
 &\hspace{-2.5cm}= 1 -  \frac1{\tilde{\mu}(U)}\tilde{\mu}\left(\cup_{(u_{k-2}', u_{k-1}')\neq (u_k-2, u_{k-1})}[u_0,\ldots, u_{k-1}, u_{k-2}',  u_{k-1}']\right)\\
& \hspace{-2.5cm} \le 1-\frac1{C_*} \tilde{\mu}\left(\cup_{(u_{k-2}', u_{k-1}')\neq (u_k-2, u_{k-1})}[u_{k-2}', u_{k-1}']\right) = 1-\frac1{C_*} \left(1-\tilde{\mu}\left([u_{k-2}, u_{k-1}]\right)\right)<1.
  \end{align*}
 (Again, we only have to consider one overlap here.)
 
 We can proceed in the same way to deal with all possible periods up to length $\lfloor\frac{k-1}2\rfloor$ within $U$ since if there is a non-empty overlap of $[u_0, \ldots, u_{k-1}] \cap[u_0, u_1, \ldots, u_i, u_0, \ldots, u_{k-1}]$, but not for smaller $i$, then this argument says that the $j$th coordinate of $U$ determines the $(i+j+1)$st coordinate of $U$ for any $i$.  But this uniquely determines all possible overlaps if $i\le \lfloor\frac{k-1}2\rfloor$, and indeed all subsequent overlaps are accounted for by the first one.   On the other hand, if there an overlap for some $i> \lfloor\frac{k-1}2\rfloor$, but not before this, there may also be multiple overlaps for subsequent $\ell>i$ which are not already accounted for at time $i$.  However, in this case we can use the quasi-Bernoulli property \ref{a5} to give  
  \begin{align*}
 \tilde{\mu}_U\left(\underline x \in U: [u_0, \ldots, u_{k-1}]\in \left\{[x_{i+1}, \ldots, x_{i+k}]\right\}_{i=\lfloor\frac{k-1}2\rfloor}^{k-1}\right) & \\
 &\hspace{-5cm}\le  C_*\left(\tilde{\mu}[u_0,\ldots, u_{\lfloor\frac{k-1}2\rfloor}] +\cdots+ \tilde{\mu}[u_0, \ldots, u_{k-2}]\right)\le \frac{C_*e^{-\left(\lfloor\frac{k-1}2\rfloor\right)}}{1-e^{-\alpha}}.
\end{align*}
 
 Since for $k$ large, this estimate is very small, we see that the the worst estimate is essentially for the first case we dealt with here, from which we would derive our new $\tilde\theta<1$. 
 
 However, we also need to deal with the cases
 \begin{align*}
 \tilde{\mu}_U\left(\underline x \in U: [u_0, \ldots, u_{k-1}]\in \left\{[x_{i+1}, \ldots, x_{i+k}]\right\}_{i=k-1}^{k+m-1}\right).
 \end{align*}
 But since these cases consist of estimates of the form
$$\tilde{\mu}_U\left(\cup_{b_0, \ldots, b_{\ell-1}\in \I} [u_0, \ldots, u_{k-1}, b_0, \ldots, b_{\ell-1}, u_0, \ldots, u_{k-1}]\right)\le C_*\tilde{\mu}(U)$$
 for $0\le \ell\le m-1$, which sum to $C_*m\tilde{\mu}(U)$.  So for $U$ small, this doesn't affect $\tilde\theta$ very much, so we abuse notation and keep this the same.  \end{proof}

 Putting together the estimates for $(II)$ we obtain $\lambda_U^{n-2-2k-2m}(1+\gamma(m))^2\tilde \theta$. Here we have used that
$$\left(\underline x \in \Sigma: [u_0, \ldots, u_{k-1}]\notin \cup_{i=2k+2m}^{n-2}[x_{i+1}, \ldots, x_{i+k}]\right)$$
is just a preimage of the set $\{\underline x \in \Sigma: \tau_U(x) \ge n-2-2k-2m\}$.
 Therefore the overall asymptotic estimate is

$$\lambda_U^{n-2-2k-2m}\left(\lambda_U^{m+k}(1-\gamma(m))-(1+\gamma(m))^2\tilde \theta\right).$$

Since $\lambda_U^{n-2-2k-2m}\ge \lambda_U^{n}$, it is sufficient to prove $\left(\lambda_U^{m+k}(1-\gamma(m))-(1+\gamma(m))^2\tilde \theta\right)>0$.  This follows since we can fix some large enough $m$ so that $\gamma(m)$ is sufficiently small, and then use \eqref{ev-bound} which, since also $\tilde{\mu}(U) \le Ce^{-\omega k}$,  implies $\lambda_U\ge 1-K_3\tilde{\mu}(U)$, so $\lambda_U^{m+k} \to 1$ as $\tilde{\mu}(U)\to 0$.

\textbf{Approximating intervals by cylinders}

Now in the case that $U$ is not a cylinder, but $\Pi(U)$ is an interval, Lemma~\ref{lem:approx} implies that it is enough to prove for certain unions of cylinders.  We will simplify the situation by just looking at a union of two cylinders, but notice that since the lemma tells us that we only have two families of cylinders, of depths $d_L+1$ and $d_R+1$ respectively, where the $d_L$-th (or $d_R$-th) prefix is the same for all members of the family, the proofs will go through for these unions too.

Suppose that $U$ and $V$ are cylinders of depth $k_U$ and $k_V$ respectively, so $U=[u_0,\ldots, u_{k_U-1}]$ and $V=[v_0, \ldots, v_{k_V-1}]$.  Assume $k_U\ge k_V$.  To make similar estimates to the single cylinder case, we will replace $k$ with $k_U$ in $(I)$ and $(II)$.  The estimate for $(I)$ follows as in the single cylinder case, so we focus on $(II)$.  To estimate the short returns we will be dealing with estimates of the form 

\begin{enumerate}
\item[(i)] $\mu_{U\cup V} \left([u_0, \ldots, u_i, u_0, \ldots, u_{k_U-1}]\right)$ for $ i\le k_U-1$;
\item[(ii)] $\mu_{U\cup V} \left([u_0, \ldots, u_i, v_0, \ldots, v_{k_V-1}]\right)$  for  $i\le k_U-1$;
\item[(iii)] $\mu_{U\cup V} \left([v_0, \ldots, v_i, u_0, \ldots, u_{k_U-1}]\right)$ for  $i\le k_V-1$;
\item[(iv)] $\mu_{U\cup V} \left([v_0, \ldots, v_i, v_0, \ldots, v_{k_V-1}]\right)$  for  $i\le k_V-1$;
\item[(v)] $\mu_{U\cup V} \left(\cup_{b_p\in \A}[v_0, \ldots, v_{k_V-1}, b_0, \ldots, b_{j-1}, u_0, \ldots, u_{k_U-1}]\right)$ and $k_V\le i\le k_U-1$;
\item[(vi)] $\mu_{U\cup V} \left(\cup_{b_p\in \A}[v_0, \ldots, v_{k_V-1}, b_0, \ldots, b_{j-1}, v_0, \ldots, v_{k_V-1}]\right)$ and $k_V\le i\le k_U-1$.
\end{enumerate}

As in the single cylinder case, for $i$ small we have to make careful estimates, but for $i$ large a rough estimate suffices (in particular, the idea for dealing with $i$ around $k+m-1$ is the same as in the single cylinder case, so we will not deal with that here).  Indeed the estimates (v) and (vi) the estimates are $\le \frac{CC_*\mu(U)\mu(V)e^{-j\omega}}{\mu(U)+\mu(V)}$, which is very small, and sum nicely.  Estimates for (i) and (iv) follow as in the simple cylinder case.
It remains to consider (ii) and (iii). 

As in the single cylinder case, the naive estimate takes over for $i$ large, so let us assume:
\begin{equation}
i\le \frac{k_V}2 \le \frac{k_U}2
\label{eq:smalli}
\end{equation}

We give some details for Case (iii). The first $i$ with non-empty overlap is $[v_0, \ldots, v_i, u_0, \ldots, u_{k_U-1}]\subset U\cup V$ with $i\le k_V-1$.

\begin{itemize}
\item 
First assume $[v_0, \ldots, v_i, u_0, \ldots, u_{k_U-1}]\subset V$.  \\Then $u_0=v_{i+1}, u_1=v_{i+2}, \ldots, u_{j} = v_{i+j+1}, \ldots, u_{k_V-i-2} =v_{k_V-1}$.

Now suppose $[v_0, \ldots, v_i, v_{i+1}, u_0, \ldots, u_{k_U-1}] \neq \es$.  Then if this set is in $V$ then $u_0=v_{i+2}=u_1=v_{i+1}, \ldots, u_j=v_{i+2+j} = u_{j+1}=v_{i+j+1}$ for $j\le \min\{k_U-1, k_V-2-i\}$, i.e., $v_{i+1}= \cdots = v_{i+j+1}=u_0$.  Hence any relevant non-empty intersection will already be contained in $[v_0, \ldots, v_i, u_0, \ldots, u_{j}] $   for $j= \min\{k_U-1, k_V-2-i\}\ge \frac{k_V}2-2$.

If instead we consider $\es\neq [v_0, \ldots, v_i, v_{i+1}, u_0, \ldots, u_{k_U-1}] \subset U$ then this also forces $u_{j+i+1}=u_j=u_{j+i+1}$ for $j\le k_U-i-2$, i.e.,  $u_0= u_1=\cdots = u_{k_U-i-2}$. Hence any relevant non-empty intersection will be contained in $[v_0, \ldots, v_i,  u_0, \ldots, u_{k_U-i-2}]$.

We can proceed as before to cover intersections  $[v_0, \ldots, v_i, v_{i+1}, \ldots, v_{i+j}, u_0, \ldots, u_{k_U-1}]$ for $i+j\le \frac{k_V}2$: we can always cover all the resulting intersections with $[v_0, \ldots, v_i, u_0, \ldots, u_\ell]$ for $\ell \ge \frac{k_V}2$.

\item
Now, instead, start by assuming  $[v_0, \ldots, v_i, u_0, \ldots, u_{k_U-1}]\subset U$, then $u_0=u_{i+1}, \ldots, u_{k_U-i-2} = u_{k_U-1}$.

If also $\es \neq [v_0, \ldots, v_i, v_{i+1},  u_0, \ldots, u_{k_U-1}]\subset U$ then as above all further relevant non-empty intersections will already be contained in $[v_0, \ldots, v_i, u_0, \ldots, u_{k_U-i-3}]$

If instead we assume $\es \neq [v_0, \ldots, v_i, v_{i+1},  u_0, \ldots, u_{k_U-1}]\subset V$ then we can deduce that $u_{i+j} = v_{i+j}$ for $0\le j\le k_V-1-i$ (we already had $u_j=v_j$ for $0\le j\le k_U-1$).  This again puts us in the single cylinder case for the first $\frac{k_V}2-1$ intersections.

Again we can proceed as before to cover all other intersections \\ $[v_0, \ldots, v_i, v_{i+1}, \ldots, v_{i+j}, u_0, \ldots, u_{k_U-1}]$ for $i+j\le \frac{k_V}2$.
\end{itemize}

Case (ii) follows similarly.
\end{proof}

\section{Hitting times estimates} \label{section:hit}

In this section, we obtain the following estimates on the expected hitting times.

\begin{theorem}\label{hit}
Assume $(f,\mu)$ satisfies \ref{a1}-\ref{a5}, \eqref{measdiam} and \eqref{bd}, so that by Proposition \ref{prop:u} we can choose $\U(\delta_3)$ which satisfies (U). There exist $0<c<C<\infty$, $\delta_3^* \in (0,\delta_3)$ such that for all $U \in \U(\delta_3^*)$,
$$\frac{c}{\mu(U)} \le \E(\tau_{U}) \le \frac{C}{\mu(U)} .$$
\end{theorem}

 In order to obtain uniform estimates on $\E(\tau_{U})$ we write $\E(\tau_{U})=\sum_{n\ge 1}\mu(\tau_{U}\ge n)$ and observe that
$$\mu(\tau_{U}\ge n)=\int  \1_{\{\tau_{U}\ge n\}}\cdot g~dm =  \int \hLp_{U}^n(   g)~dm$$
and so it will be sufficient to obtain uniform estimates on $ \int \hLp_{U}^n(   g)~dm$.

\begin{lemma}\label{uniformhit}  There exist constants $0<c<C<\infty$ and $0<\delta_2^*\le \delta_1$ such that for all $U \in \U(\delta_2^*)$ and $n \in \N$,
 \begin{align*}
c\lambda_{U}^n \le \int \hLp_{U}^n(   g)~dm \le C\lambda_{U}^n.
\end{align*}
\end{lemma}

\begin{proof}
By Corollary \ref{kl}, for all $U \in \U(\delta_1)$.
\begin{align*}
\int \hLp_{U}^n(g) dm&=\lambda_{U}^n\left( \int \Pi_{U}(g) + \lambda_{U}^{-n}R_{U}^n(g) dm\right) = \lambda_{U}^n\left( \int c_{U}g_{U}+ \lambda_{U}^{-n}R_{U}^n(g) dm\right)\\
&=\lambda_{U}^n\left(  c_{U} + \lambda_{U}^{-n}\int R_{U}^n(g) dm\right)
\end{align*}
where we have used that $\int g_{U}dm=1$. By \eqref{k2},
$$\left| \int R_{U}^n(g)dm\right| \le \norm{R_{U}^n(g)} \le K_2 \beta_1^n\norm{g}.$$
By \eqref{k1}, 
$$|1-c_{U}|=\left|\int\Pi(g)dm-\int \Pi_U(g)dm\right| \le K_1 m( U)^\eta\norm{g}.$$
Now, we can choose $N \in \N$  so that $K_2\norm{g}(\beta_1/\beta_2)^N<1$. Now choose $\delta_2^* \le \delta_1$ such that  $M:= \sup_{U \in \U(\delta_2^*)}K_1\norm{g} m( U)^\eta<1-K_2\norm{g}(\beta_1/\beta_2)^N$. Then for all $n \in \N$ and $U \in \U(\delta_2^*)$,
\begin{align*}
\int \hLp_{U}^n(g) dm&\le\lambda_{U}^n\left(1+M+K_2(\beta_1/\beta_2)^n\norm{g}\right) \\
&\le \lambda_{U}^n\left(1+M+K_2\norm{g}\right) .
\end{align*}
Moreover, since $\mathrm{ess} \inf g=\frac{1}{C_1}>0$ and $|g_U|_\infty\leq K_4$, one has $g \geq \frac{1}{C_1 K_4}g_U$ so that for any $n \geq 0$,
\begin{align*}
\int \hLp_{U}^n(g) dm=\int_{\hU^{n-1}} g d m \geq \frac{1}{C_1K_4} \int_{\hU^{n-1}} g_U dm=\frac{1}{C_1K_4} \int \hLp_U^n g_U dm=\frac{1}{C_1K_4}\lambda_U^n.
\end{align*}
\end{proof}

Therefore, to complete the proof of Theorem \ref{hit}, we require estimates on $\lambda_U$. This is provided by the following result.

\begin{proposition} \label{eig-bound}
There exist constants $0<c<C<\infty$ and $\delta_3 \in (0,\delta_2)$  such that for all  $U \in \U(\delta_3)$,
$$c\mu(U) \le 1- \lambda_U \le C \mu(U).$$
\end{proposition}

\begin{proof}

The upper bound follows from \eqref{ev-bound}. For the lower bound we adapt some arguments from \cite{kl2}. In  \cite{kl2} and \cite{BruDemTod18} the limit $\lim_{\delta \to 0} \frac{1-\lambda_{U_\delta}}{\mu(U_\delta)}$ is computed for $\bigcap_{\delta>0}U_\delta =\{z\}$. We require something less precise (only a lower bound on the ratio $\frac{1-\lambda_{U_\delta}}{\mu(U_\delta)}$), but it needs to be uniform over $U \in \U(\delta_3)$ for some $\delta_3>0$.  

We begin by verifying that uniform versions of (A1)-(A6) from \cite{kl2} hold over our family of operators $\{\hLp_U\}_{U \in \U(\delta_1)}$. (A1)-(A2) is precisely the spectral decomposition from Corollary \ref{decomp}. (A3) requires that
$$\sum_{n=0}^\infty \sup_{U \in \U(\delta_1)} \frac{\norm{R_U^n}}{\lambda_U^n}<\infty$$
which holds by \eqref{k2}.  (A4) requires that $\int g_Udm=1$, which we have, and that
$$\sup_{U \in \U(\delta_1)}\norm{g_U}<\infty$$
which follows from \eqref{k4}.
  (A5) requires that $\sup_{U \in \U(r)} \eta_U \to 0$ where
\begin{equation}\label{eta}
\eta_U:= \sup_{\norm{\psi}=1} \left|\int (\Lp \psi-\hLp_U\psi)dm\right|,
\end{equation}
which holds by Lemma \ref{pert} since $\eta_U \le m(U)$. Finally (A6) requires that
$$\eta_U\norm{(\Lp-\hLp_U)(g)} \le C'\mu(U)$$
for some constant $C'$ which doesn't depend on $U \in \U(\delta_1)$. 
Since
$$\norm{(\Lp-\hLp_U)(g)}=\norm{\Lp(\1_Ug)} \le \norm{\Lp}\norm{\1_Ug} \le\norm{\Lp}\norm{g}$$
we indeed have
$$\eta_U\norm{(\Lp-\hLp_U)(g)}\le m(U)\norm{\Lp}\norm{g} \le C_1\norm{\Lp}\norm{g} \mu(U)$$
as desired.

The above implies that \cite[Lemma 6.1]{kl2} holds for a uniform constant $C$. In particular, defining for each $n \in \N$ 
\begin{equation}\label{kappa}
\kappa_n:=K_2 \sum_{k=n}^\infty  \left(\frac{\beta_1}{\beta_2}\right)^n,
\end{equation}
 there exists $C>0$ such that for all $U \in \U(\delta_1)$,
\begin{enumerate}
\item [(a)] $|1-\int gdm_U| \le Cm(U)$ where $m_U$ is given by Corollary \ref{decomp}(a),
\item [(b)] $\frac{\norm{R_U^ng}}{\lambda_U^n} \le C\kappa_n(\norm{(\Lp-\hLp_U)(g)}+|1-\lambda_U|)$.
\end{enumerate}

Following the proof of Theorem 2.1 in \cite{kl2} now yields that for any $n \in \N$,
\begin{equation} \label{kl-eqn}
\frac{1-\lambda_U}{\mu(U)}=\frac{\left(1-\sum_{k=0}^{n-1}\lambda_U^{-k}q_{k,U} \right)+\kappa_n'}{(1+O(\eta_U))(1+nO(\eta_U))}
\end{equation}
where $\kappa_n'=O(\kappa_n)$ and
$$q_{k,U}=\frac{m((\Lp-\hLp_U)\hLp_U^k(\Lp-\hLp_U)g)}{m((\Lp-\hLp_U)g)}=\frac{\mu(E_U^k)}{\mu(U)}$$
for
$$E_U^k:=\{x \in U: f^i(x) \notin U, \; i=1, \ldots, k, \; f^{k+1}(x) \in U\},$$
see  \cite{BruDemTod18}. 

Recall that by $\kappa_n'=O(\kappa_n)$ we mean that for some constant $C''>0$ which is independent of $n$ (and $U$), $|\kappa_n'| \le C''|\kappa_n|$. Let $C$ be any constant such that 
$$|\kappa_n'|\le  C (\beta_1/\beta_2)^n$$
for all $ n \in \N$. Choose $\delta_2' \in (0,\delta_2]$ such that $$\rho:= \frac{\beta_1}{\beta_2\inf_{U \in \U(\delta_2')}\lambda_U}<1,$$
which is possible by \eqref{ev-bound}, and such that $\inf_{U \in \U(\delta_2')}\lambda_U\ge \beta_2$. Let $c$ be given by Proposition \ref{prop:RTS}. Let $N$ be sufficiently large that $C\rho^N<c\beta_2$. We claim that there exists $c'>0$ and $0<\delta_2'\le \delta_2$ such that for all $U \in \U(\delta_2')$ with $N_U:=1/\mu(U)>N$,
\begin{equation} \label{q-eqn}
1-\sum_{k=0}^{N_U-1} \frac{\mu(E_U^k)}{\mu(U)} +\kappa_{N_U}' \ge c'.
\end{equation}
To see this note that the left hand side of \eqref{q-eqn} equals
\begin{align*}
1-\sum_{k=0}^{N_U-1} \frac{\mu(E_U^k)}{\mu(U)} +\kappa_{N_U}' &=1-\frac{\mu(x \in U: \tau_U(x) \le N_U)}{\mu(U)}+\kappa_{N_U}' = 1-\mu_U(\tau_U \le N_U)+\kappa_{N_U}'\\
&= \mu_U(\tau_U\ge N_U+1)+\kappa_{N_U}'.
\end{align*}
By Proposition \ref{prop:RTS}, $\mu_U(\tau_U\ge N_U+1)\ge c\lambda_U^{N_U+1}$ therefore 
$$\mu_U(\tau_U\ge N_U+1)+\kappa_{N_U}' \ge c\lambda_U^{N_U+1}+\kappa_{N_U}' \ge \lambda_U^{N_U}(c\lambda_U-C\left(\beta_1/\beta_2\lambda_U\right)^{N_U}).$$
 All of this implies that for all $U \in \U(\delta_2')$,
$$1-\sum_{k=0}^{N_U-1} q_{k,U} +\kappa_{N_U}' \ge \lambda_U^{N_U}(c\beta_2-C\rho^{N_U})>0,$$ 
proving \eqref{q-eqn}. Finally, we choose $0<\delta_3\le \delta_2'$ sufficiently small that the denominator in \eqref{kl-eqn}, 
$$\inf_{U \in \U(\delta_3)}(1+O(\eta_U))(1+N_UO(\eta_U))>0$$
(possible since $\eta_U \le 1/m(U)$ and noting that it is clearly uniformly bounded above too) and such that the numerator 
$$\inf_{U \in \U(\delta_3)}\left(1-\sum_{k=0}^{N_U-1}\lambda_U^{-k}q_{k,U} \right)+\kappa_{N_U}'>0$$
which is possible by \eqref{q-eqn} and \eqref{ev-bound}.  By \eqref{kl-eqn} this completes the proof of the proposition.
\end{proof}

\begin{proof}[Proof of Theorem \ref{hit}]
Since 
$$\E(\tau_{U})=\sum_{n\ge 1}\mu(\tau_{U}\ge n)=\sum_{n\ge 1}\int \hLp_{U}^n(   g)~dm,$$
it follows from Lemma \ref{uniformhit} that for all $r \in (0,\delta_2]$ and $U \in \U(\delta_2^*)$,
$$\frac{c}{1-\lambda_{U}} \le \E(\tau_{U}) \le \frac{C}{1-\lambda_{U}}$$
for some uniform constants $0<c<C$. Take $\delta_3^*=\min\{\delta_2^*, \delta_3\}$ where $\delta_3$ is given by Proposition \ref{eig-bound}. The proof now follows from Proposition \ref{eig-bound}.
\end{proof}

\section{Expected cover time: uniformly hyperbolic case} \label{section:matthews}

In this section we prove Theorems \ref{cover} and \ref{cover2}. This will be done by generalising the `Matthews method' for Markov chains (see \cite{matthews} or  \cite[Chapter 11]{peres}) which will establish a relationship between the expected cover time and the expected hitting times, at which point we can apply our estimates from \S \ref{section:hit}. Recall from \S \ref{section:uniform} that there exists a subshift $\sigma:\Sigma \to \Sigma$ and a projection $\Pi:\Sigma \to I$ such that $f\circ \Pi=\Pi \circ \sigma$. Further recall that there exists a quasi-Bernoulli measure $\tmu$ on $\Sigma$ such that $\Pi_*\tmu=\mu$. In this section we will primarily work with symbolic versions of cover and hitting times, which we describe below.

Let $\mathcal{P}$ be a finite set of subsets of $\Sigma$ where (a) for each $P \in \mathcal{P}$, there exists a finite or countable set of words $P^* \subset \Sigma^*$ such that $P=\bigcup_{\i \in P^*} [\i]$, and (b) for distinct $P, Q \in \mathcal{P}$, $P \cap Q=\emptyset$.\footnote{Note that here we only require disjointness as subsets of $\Sigma$, rather than disjointness of their $\Pi$-projections to $\R$.} Note that $\mathcal{P}$ is not necessarily a partition of $\Sigma$ since there is no requirement that $\bigcup_{P \in \mathcal{P}} P=\Sigma$. We will call $\mathcal{P}$ a subpartition of $\Sigma$ if it satisfies (a) and (b). For $\i \in \Sigma$ we let $\tau_{\mathcal{P}}(\i)$ denote the first time $n$ that $\{\i, \sigma(\i), \ldots, \sigma^n(\i)\}$ has visited every element in $\mathcal{P}$. For $P \in \mathcal{P}$ we let $\tau_P: \Sigma \to \N_0 \cup \{\infty\}$ be defined as 
$$\tau_P(\i)= \inf\{n \geq 0:  \sigma^n(\i) \in P\} $$
i.e. the first $n \geq 0$ such that $\sigma^n(\i)$ begins with a word in $P^*$.

\begin{remark} \label{hitesc}
Since $\tau_P$ may take the value 0, strictly speaking $\tau_P$ is not a hitting time in the same sense as \eqref{hitting}. Let us briefly write $\tilde{\tau}_P(\i)= \inf\{n \geq 1:  \sigma^n(\i) \in P\} $. Then
$$\{\i \in \Sigma: \tilde{\tau}_P(\i) =n\}=\sigma^{-1}(\{\i \in \Sigma: \tau_P(\i)=n-1\}),$$
implying in particular that $\E_{\tilde{\mu}}(\tau_P)=\E_{\tilde{\mu}}(\tilde{\tau}_P)$. Hence, slightly abusing notation, we will still refer to $\tau_P$ as the hitting time throughout this section. The fact that the expectations coincide along with the fact that $\Pi_*\tmu=\mu$, $\Pi \circ \sigma=f\circ \Pi$ and non-uniquely coded points have zero measure, we have $\E_{\tmu}(\tau_P)=\E_\mu(\tau_{\Pi(P)})$.
\end{remark}

We need to set up some notation and obtain some preliminary results. Let $\mathcal{P}=\{P_1, \ldots, P_N\}$ be a subpartition of $\Sigma$ and given a permutation $s$ of $\{1, \ldots, N\}$ and $2 \le k \le N$, let $A_{s, k} \subset \Sigma$ be the set of points which visit $P_{s(k)}$ for the first time after $P_{s(1)}, \ldots, P_{s(k-1)}$ have all been visited. That is, denoting $\tau_{s(k)} \equiv \tau_{P_{s(k)}}$ and $\tau_{s(1)}^{s(k-1)}$ to be the first time $n \in \N_0$ that $P_{s(1)}, \ldots, P_{s(k-1)}$ have all been visited by $\{\i, \sigma(\i),\ldots, \sigma^n(\i)\}$, we have
$$\tau_{s(k)}(\i)> \tau_{s(1)}^{s(k-1)}(\i)$$
for $\i \in A_{s, k}$.

Let $\mathbb{P}$ be the uniform measure on the set of all permutations $s \in S_N$ of $\{1, \ldots, N\}$.

\begin{lemma} \label{A_average}
Let $\mathcal{P}$ be a subpartition of $\Sigma$ and $A_{s,k}$ and $\mathbb{P}$ be as above. Then
\begin{equation}
\int \tmu(A_{s,k}) d\mathbb{P}=\frac{1}{k}. \label{p}
\end{equation}
\end{lemma}

\begin{proof}
 Fix $2 \le k \le N$. For each $s \in S_N$ consider the unordered set $\{s(1), \ldots, s(k)\}$. Note that there are $\frac{N(N-1) \cdots (N-(k-1))}{k!}$ possible values. For each possible value $\{i_1, \ldots, i_k\} \subset \{1, \ldots, N\}$ that this set can take, let $S_N(\{i_1, \ldots, i_k\})$ denote the set of all $s$ for which $\{s(1), \ldots, s(k)\}=\{i_1, \ldots, i_k\}$, thinking of these as unordered sets.

Next, we can further separate each $S_N(\{i_1, \ldots, i_k\})$ into $k$ subsets $S_N^{i_j}(\{i_1, \ldots, i_k\})$, $(1 \le j \le k)$, which determines the set of all $s \in S_N(\{i_1, \ldots, i_k\})$ for which $s(k)=i_j$. Note that each $S_N^{i_j}(\{i_1, \ldots, i_k\})$ contains $(N-k)!(k-1)!$ permutations, corresponding to $(N-k)!$ ways to order the last $N-k$ terms and and $(k-1)!$ ways to arrange the first $k-1$ terms. Over each $s \in S_N^{i_j}(\{i_1, \ldots, i_k\})$, the set $A_{s,k}$ is constant. If for each $1 \le j\le k$ we choose a representative $s_j \in S_N^{i_j}(\{i_1, \ldots, i_k\})$ then since the sets in $\P$ are pairwise disjoint, it follows that $\{A_{s_j,k}\}_{j=1}^k$ are pairwise disjoint and $\bigcup_{j=1}^k A_{s_j,k}=\Sigma$.

Hence for any choice of $\{i_1, \ldots, i_k\} \subset \{1, \ldots, N\}$,
\begin{eqnarray}
\int_{S_N(\{i_1, \ldots, i_k\})} \tmu(A_{s,k})d \mathbb{P}=\sum_{n=1}^k \frac{(k-1)!(N-k)!}{N!} \tmu(A_{s_j,k}),
\label{part-integral}
\end{eqnarray}
where the factor $\frac{1}{N!}$ comes from the fact that $\mathbb{P}$ is uniformly distributed. Now, since $\bigcup_{j=1}^k A_{s_j,k}=\Sigma$ and  $\{A_{s_j,k}\}_{j=1}^k$ are pairwise disjoint we have
$$\tmu(A_{s_k,k})= 1-\sum_{j=1}^{k-1} \tmu(A_{s_j,k})$$
and substituting this into \eqref{part-integral} we obtain
\begin{align*}
\int_{S_N(\{i_1, \ldots, i_k\})}\!\! \tmu(A_{s,k})d \mathbb{P}&= \sum_{j=1}^{k-1}\frac{(k-1)!(N-k)!}{N!} \tmu(A_{s_j,k}) +\frac{(k-1)!(N-k)!}{N!} \Big(1-\sum_{j=1}^{k-1} \tmu(A_{s_j,k})\Big) \\
&=\frac{(k-1)!(N-k)!}{N!}.
\end{align*}
Therefore,
$$\int_{S_N} \tmu(A_{s,k}) d\mathbb{P}=\frac{N(N-1) \cdots (N-(k-1))}{k!} \cdot \frac{(k-1)!(N-k)!}{N!}=\frac{1}{k}.$$
\end{proof}

We are now ready to prove Theorems \ref{cover} and \ref{cover2}. Note that the upper bound in Theorem \ref{cover} is just a special case of the upper bounds in Theorem \ref{cover2} (by using \eqref{eq:m_scale} and that $g=d\mu/dm$ is strictly positive on $\Lambda$), hence it suffices to prove the upper bound in Theorem \ref{cover2}.

\begin{proof}[Proof of upper bound in Theorem \ref{cover2}]
For each $\delta>0$ consider the set of closed intervals $\U_\delta$ given by (U). Recall that for any $U \in \U_\delta$, there exists $x \in \Lambda$ such that $B(x,t\delta) \subseteq U \subseteq B(x,T\delta)$ for some uniform constants $t$ and $T$. Let $\mathcal{P}_\delta$ be the subpartition  
\begin{equation} \label{Pd}
\mathcal{P}_\delta=\left\{ \mathrm{int}(\Pi^{-1}(U))\;:\; U \in \U_\delta\right\}
\end{equation}
where $\mathrm{int}$ is taking the interior inside $\Sigma$. This ensures $\mathcal{P}_\delta$ is a subpartition (by ensuring a set in $\mathcal{P}_\delta$ cannot contain any isolated points in the cylinder sets topology and guarantees pairwise disjointness of sets in $\mathcal{P}_\delta$).

By  \ref{Ua} and \ref{Uc}, if $\{\i, \sigma(\i), \ldots, \sigma^n(\i)\}$ has visited each element in $\P_{\delta/2T}$ then $\{x, f(x), \ldots, f^n(x)\}$ is $\delta$-dense, hence 
\begin{equation}\label{ballstocyl}
\E_\mu(\tau_{\delta}) \le \E_{\tmu} (\tau_{\mathcal{P}_{\delta/2T}}).
\end{equation}

Now, fix $\delta$ and fix $\P=\mathcal{P}_{\delta/2T}$. Let $L=\max_{\i \in \P} |\i|$, so that $L=O(\log(1/\delta))$ by \ref{Ue}. We order $\P=\{P_1, \ldots, P_N\}$ and let $\mathbb{P}$ be the uniform measure on the set of permutations $S_N$. We have

\begin{eqnarray}
\int \tau_{\mathcal{P}} d\tmu&=& \int \int \tau_{s(1)}^{s(N)} d\tmu d \mathbb{P} =  \int \int \tau_{s(1)}d \tmu d \mathbb{P}+\sum_{k=2}^N \int \int \tau_{s(1)}^{s(k)}-\tau_{s(1)}^{s(k-1)}d \tmu d \mathbb{P} \nonumber\\
&=& \int \int \tau_{s(1)}d \tmu d \mathbb{P}+ \sum_{k=2}^N \int \int_{A_{s, k}}\tau_{s(k)}-\tau_{s(1)}^{s(k-1)} d\tmu d \mathbb{P}. \label{diagonal}
\end{eqnarray}

Fix $s \in S_N$. Let $A_{s,k}^*$ denote the set of all finite words $\i=i_0\ldots i_n \in \Sigma^*$ for which (a) for all $1 \le j \le k-1$, $\i$ contains at least one word from $P_{s(j)}$, and we let $0 \le n_j \le n-L$ denote the index at which the first word from $P_{s(j)}$ begins (so that $i_{n_j}$ is the first digit of this word) (b) for $\ell \le \max_{1 \le j \le k-1}n_j$, $\i$ does not contain any word from $P_{s(k)}$ beginning at $i_\ell$ and (c) $|\i|=L+\max_{1 \le j \le k-1}n_j$. 
Note that $A_{s,k}=\bigcup_{\i \in A_{s,k}^*} [\i]$. For $\i \in A_{s,k}^*$ let $B_{\i,s,k}$ denote the set of all finite words $\j \in \Sigma^*$ for which $\i\j \in \Sigma^*$ and $\j$ ends in the first occurrence of a word from $P_{s(k)}$. Given $\j \in B_{\i,s,k}$ let $\j_{s(k)} \in P_{s(k)}$ denote the word that $\j$ ends in. Also, let $B_{s,k}$ denote the set of all finite words $\j \in \Sigma^*$ which end in the first occurrence of a word from $P_{s(k)}$. Then
$$\int_{A_{s,k}}\tau_{s(k)}-\tau_{s(1)}^{s(k-1)}d\tmu \le \sum_{\i \in A_{s,k}^*} \sum_{\j \in B_{\i,s,k}} \tmu([\i\j])(|\j|+L-|\j_{s(k)}|).$$
Then by the quasi-Bernoulli property \ref{a6} we have
\begin{eqnarray*}
 \sum_{\i \in A_{s,k}^*} \sum_{\j \in B_{\i,s,k}} \tmu([\i\j])(|\j|+L-|\j_{s(k)}|) &\le & C_*\sum_{\i \in A_{s,k}^*} \sum_{\j \in B_{\i,s,k}} \tmu([\i])\tmu([\j]) (|\j|+L-|\j_{s(k)}|) \\
&\le& C_*\tmu(A_{s,k})  \sum_{\j \in B_{s,k}} \tmu([\j]) (|\j|+L-|\j_{s(k)}|)\\
&\le& C\tmu(A_{s,k}) ( \E_{\tmu}(\tau_{s(k)}+2L)).
\end{eqnarray*}
To see the final inequality, we use that $|\tau_{s(k)}''-\tau_{s(k)}| \le L$ where, denoting $\tau_{s(k)}'(\i)$ to be the first time $n \in \N$ that $\i|_n$ ends in a word from $P_{s(k)}$, $\tau_{s(k)}''(\i)$ denotes the time $n' \le n$ that this word from $P_{s(k)}$ begins.
Putting back into \eqref{diagonal} and using $L=O(\log(1/\delta))$ we get
\begin{align*}
\int \tau_{\mathcal{P}} d\tmu &\le C( \log(1/\delta)+ \max_{1 \le m \le N} \E_{\tmu}(\tau_{m}))\left(1+ \sum_{k=2}^N \int \tmu(A_{s,k}) d\mathbb{P}\right)\\
& =C \left( \log(1/\delta)+ \max_{1 \le m \le N} \E_{\mu}(\tau_{\Pi(P_m)})\right) (1+ 1/2+ \cdots +1/N)
\end{align*}
by Lemma \ref{A_average} and Remark \ref{hitesc}. Note that $N \le C/\delta$. Using this, we obtain that
\begin{align*}
\E_{\mu}(\tau_\delta) & \le C\left( \max_{U \in \U_{\delta/2T}} \frac{1}{\mu(U)} \log\left(1/\delta\right) + (\log(1/\delta))^2\right)\\
&\le C\left( \frac{1}{\min_{x \in \Lambda}\mu(B(x,t\delta/2T))} \log\left(1/\delta\right) + (\log(1/\delta))^2\right) \le\frac{C}{M_\mu(t\delta/2T)} \log\left(1/\delta\right)  
\end{align*}
where in the first inequality we have used \eqref{ballstocyl}, Theorem \ref{hit} and the fact that $\mu$ is non-atomic, and in the second we have used \ref{Ua}. In the specific setting that  \eqref{eq:m_scale} holds, similarly 
$$\E_{\mu} (\tau_\delta)  \le C\left( \max_{U \in \U_{t\delta/2T}} \frac{1}{\mu(U)} \log\left(1/\delta\right) + (\log(1/\delta))^2\right) \le\frac{C}{\delta^{s_f} }\log\left(1/\delta\right)  $$
where again we have used \ref{Ua} to conclude that $ \max_{U \in \U_{\delta/2T}} \frac{1}{\mu(U)} \le \max_{x \in \Lambda} \frac{1}{\mu(B(x,t\delta/2T))}$ and  \eqref{eq:m_scale} for the second inequality.
\end{proof}

Now we prove the lower bounds in Theorems \ref{cover} and \ref{cover2}. The lower bound in the first displayed equation in Theorem \ref{cover} is a special case of the lower bound in Theorem \ref{cover2}. Therefore, we will first prove the lower bound from Theorem \ref{cover2} and then proceed to prove the lower bound from the second displayed equation in Theorem \ref{cover}.  

\begin{proof}[Proof of lower bounds in Theorem \ref{cover} and \ref{cover2}]
We begin by proving the lower bound from Theorem \ref{cover2}. Define $P_\delta$ as in \eqref{Pd}. Note that  if $\{x, f(x), \ldots, f^n(x)\}$ is $\delta$-dense then it necessarily visits every ball centred at a point in $\Lambda$ of radius greater than or equal to $\delta$ hence $\{\i, \sigma(\i), \ldots, \sigma^n(\i)\}$ has visited each element in $\mathcal{P}_{2\delta/t}$ by \ref{Ua} and \ref{Ub}.  Hence
\begin{equation*}
\E_\mu(\tau_{\delta}) \ge \E_{\tmu}(\tau_{\mathcal{P}_{2\delta/t}}).
\end{equation*}
Moreover, for each $P \in \mathcal{P}_{2\delta/t}$,
$$\E_{\tmu}(\tau_{\mathcal{P}_{2\delta/t}}) \ge \E_{\tmu} (\tau_P)=\E_\mu (\tau_{\Pi(P)}) \ge \frac{c}{\mu(\Pi(P))}$$
where we have used Remark \ref{hitesc} to justify the equality of expectations. Since some $U=\Pi(P) \in \U_{2\delta/t}$ must belong to the ball of minimum measure at scale $2T\delta/t$ we have
$$\E_\mu(\tau_{\delta})\ge \frac{c}{\min_{x \in \Lambda} \mu(B(x, 2T\delta/t))}=\frac{c}{M_\mu(2T\delta/t)}$$
proving the lower bound from Theorem \ref{cover2}.

Next, we prove the lower bound which appears in the second displayed equation in Theorem \ref{cover}. In particular, we are now additionally assuming \eqref{eq:m_scale} holds, $f$ is Markov and $f$ has at least 2 full branches, moreover Proposition \ref{prop:u'} is applicable. Since $f$ is Markov, it is useful to keep in mind that for $\i, \j \in \Sigma^*$ the legality of concatenations $\i\j$ will be equivalent to the legality of concatenating the last digit of $\i$ with the first digit of $\j$. For each $\delta$ sufficiently small consider the set $\mathcal{V}_\delta$ given by Proposition \ref{prop:u'}. Define 
$$\mathcal{Q}_\delta=\left\{[wab^{n_3}]\;: \Pi([wab^{n_3}]) \in \mathcal{V}_\delta\right\} \subset \{[wab^{n_3}]: w \in \{a,b\}^*\}$$
noting that $\mathcal{Q}_\delta$ is a subpartition since sets in $\V_\delta$ are pairwise disjoint. Note that  if $\{x, f(x), \ldots, f^n(x)\}$ is $\delta$-dense then it necessarily visits every ball centred at a point in $\Lambda$ of radius greater than or equal to $\delta$ hence by \ref{Ua} and \ref{Uc} $\{\i, \sigma(\i), \ldots, \sigma^n(\i)\}$ has visited each element in $\Q_{2\delta/t}$.  Hence
\begin{equation}\label{ballstocyl2}
\E_\mu(\tau_{\delta}) \ge \E_{\tmu} (\tau_{\mathcal{Q}_{2\delta/t}}).
\end{equation}

Now, fix $\delta$ and fix $\Q=\mathcal{Q}_{2\delta/t}$. Write $\Q=\{[\i_1], \ldots,[ \i_M]\}$ and let $\mathbb{P}$ be the uniform measure on the set of permutations $S_M$.
Let $A_{s, k} \subset \Sigma$ be the set of points which visit $[\i_{s(k)}]$ for the first time after $[\i_{s(1)}], \ldots, [\i_{s(k-1)}]$ have all been visited. For brevity we write $\tau_{s(k)} \equiv \tau_{[\i_{s(k)}]}$ and $\tau_{s(1)}^{s(k-1)}$ to be the first time that $[\i_{s(1)}], \ldots, [\i_{s(k-1)}]$ have all been visited. Then as in \eqref{diagonal}
\begin{eqnarray}
\int \tau_{\mathcal{Q}} d\tmu= \int \int \tau_{s(1)}d \tmu d \mathbb{P}+ \sum_{k=2}^M \int \int_{A_{s, k}}\tau_{s(k)}-\tau_{s(1)}^{s(k-1)} d\tmu d \mathbb{P}.\label{diagonal2}
\end{eqnarray}
Fix $s \in S_M$. Let $A_{s,k}^*$ denote the set of all finite words $\i \in \Sigma^*$ for which (a) $\i$ contains every word in $\{\i_{s(1)}, \ldots, \i_{s(k-1)}\}$ (b) $\i$ does not contain the word $\i_{s(k)}$ and (c) for some $1 \le i \le k-1$, $\i$ ends in the first occurrence of the word $\i_{s(i)}$. In particular, $A_{s,k}=\bigcup_{\i \in A_{s,k}^*} [\i]$ by  Proposition \ref{prop:u'}(g).

Let $B_{\i,s,k}$ denote the set of all finite words $\j \in \Sigma^*$ for which $\i\j$ ends in the first occurrence of the word $\i_{s(k)}$. Note that since $\i$ ends with a digit in $\{a,b\}$ and the system is Markov, necessarily $\i\j \in \Sigma^*$. Then by the (lower) quasi-Bernoulli property \ref{a6},
\begin{eqnarray}
\int_{A_{s,k}}\tau_{s(k)}-\tau_{s(1)}^{s(k-1)}d\tmu&=& \sum_{\i \in A_{s,k}^*} \sum_{\j \in B_{\i,s,k}} \tmu([\i\j])(|\i_{s(i)}|+|\j|-|\i_{s(k)}|) \nonumber\\
 &\ge &1/C_*\sum_{\i \in A_{s,k}^*} \sum_{\j \in B_{\i,s,k}} \tmu([\i])\tmu([\j]) (|\i_{s(i)}|+|\j|-|\i_{s(k)}|)\nonumber \\
&\ge& \tmu(A_{s,k})/C_* \inf_{\i \in A_{s,k}^*} \sum_{\j \in B_{\i,s,k}} \tmu([\j]) (|\i_{s(i)}|+|\j|-|\i_{s(k)}|) . \label{hitjump}
\end{eqnarray}
Note that
\begin{equation}\label{lb-diag}
\E_{\i_{s(i)}}(\tau_{s(k)})=\sum_{\j \in B_{\i,s,k}} \frac{\tmu([\i_{s(i)}\j])}{\tmu([\i_{s(i)}])} (|\i_{s(i)}|+|\j|-|\i_{s(k)}|) \le  C_*\sum_{\j \in B_{\i,s,k}} \tmu([\j]) (|\i_{s(i)}|+|\j|-|\i_{s(k)}|),
\end{equation}
where we are using $\E_{\i_{s(i)}}$ to denote the expectation (with respect to $\tmu$) conditioned to $[\i_{s(i)}]$. 

Let $L =\max_{\i \in \Q} |\i|$, noting that $L=O(\log(1/\delta))$ by Proposition \ref{prop:u'}(f). Consider the function
$$w_{s(k),n}(x):=\inf\{m \ge n: \sigma^{m}(x) \in [\i_{s(k)}]\}.$$
Denote $w_{s(k),s(i)}=w_{s(k),|\i_{s(i)}|}$. Observe that on $[\i_{s(i)}]$,
$$w_{s(k),s(i)} \le \tau_{s(k)}+\sum_{\ell=1}^L \mathbf{1}_{\{\tau_{s(k)}= |\i_{s(i)}|+\ell-|\i_{s(k)}|\}} (w_{s(k), |\i_{s(i)}|+\ell}+\ell),$$
therefore
\begin{equation}\E_{\i_{s(i)}}(w_{s(k),s(i)}) \le \E_{\i_{s(i)}}(\tau_{s(k)})+\sum_{\ell=1}^L \E_{\i_{s(i)}}(\mathbf{1}_{\{\tau_{s(k)}= |\i_{s(i)}|+\ell-|\i_{s(k)}|\}} (w_{s(k), |\i_{s(i)}|+\ell}+\ell)). \label{condeqn} \end{equation}

Let $C_{s(i),s(k),\ell} \subset \Sigma_{\ell}$ be the set of words $\k$ such that $\i_{s(i)}\k$ ends with the first occurence of the word $\i_{s(k)}$ (note that necessarily $\i_{s(i)}\k \in \Sigma^*$ since $\i_{s(i)} \in \{a,b\}^*$) and let $C_{s(i),s(k)}=\bigcup_{\ell=1}^{L} C_{s(i),s(k),\ell}  $.  Let $D_{s(k)}$ denote the set of words $\j$ such that $\j$ ends with the first occurrence of the word $\i_{s(k)}$. Define 
$$P_{s(i),s(k)} = \sum_{\mathbf{\k}\in C_{s(i),s(k)}} \tmu([\k]) \;\; \textnormal{and} \; \;P_{s(i),s(k),\ell} = \sum_{\mathbf{\k}\in C_{s(i),s(k),\ell}} \tmu([\k])  $$
Note that for any choice of permutation $s$, $P_{s(i),s(k)} \le C_*\tmu[ab^{n_3}]K$. To see this, write $\i_{s(k)}=u_1 \ldots u_n ab^{n_3}$, thus
\begin{align*}
P_{s(i),s(k)} &\le \tmu[ab^{n_3}]+\tmu[u_nab^{n_3}]+ \cdots + \tmu[u_1\ldots u_n ab^{n_3}] \le C_*\tmu[ab^{n_3}]K.
\end{align*}
In particular $r:=C_*^3P_{s(i),s(k)} <1$ by Proposition \ref{prop:u'}(a).

\begin{align*}
\E_{\i_{s(i)}}\left(\mathbf{1}_{\{\tau_{s(k)}= |\i_{s(i)}|+\ell-|\i_{s(k)}|\}} (w_{s(k), |\i_{s(i)}|+\ell}+\ell)\right)
&= \sum_{\k \in C_{s(i),s(k),\ell}}\sum_{\j\in D_{s(k)}} \frac{\tmu([\i_{s(i)}\k\j])}{\tmu([\i_{s(i)}])} (\ell+|\j|-|\i_{s(k)}|) \\
&\hspace{-1cm}\le C_*^3 \sum_{\k \in C_{s(i),s(k),\ell}}\tmu([\k])\sum_{\j\in D_{s(k)}} \frac{\tmu([\i_{s(i)}\j])}{\tmu([\i_{s(i)}])} (\ell+|\j|-|\i_{s(k)}|) \\
&\hspace{-1cm}\le C_*^3 P_{s(i),s(k),\ell} \E_{\i_{s(i)}}(w_{s(k),s(i)}+L).
\end{align*}
So by \eqref{condeqn}
\begin{equation}
\E_{\i_{s(i)}}(\tau_{s(k)}) \ge \E_{\i_{s(i)}}(w_{s(k),s(i)})(1- C_*^3P_{s(i),s(k)})-C_*^3P_{s(i),s(k)}L. \label{condeqn2}\end{equation}
Moreover
\begin{align}
\E_{\i_{s(i)}}(w_{s(k),s(i)})&= \sum_{\j \in D_{s(k)}} \frac{\tmu([\i_{s(i)}\j])}{\tmu([\i_{s(i)}])} (|\j|-|\i_{s(k)}|) \nonumber\\
&\ge \frac{1}{C_*} \sum_{\j \in D_{s(k)}} \tmu([\j]) (|\j|-|\i_{s(k)}|)=\frac{1}{C_*} \E_{\tmu} (\tau_{s(k)}). \label{uncond}
\end{align}
So by \eqref{lb-diag}, \eqref{condeqn2} and \eqref{uncond},
\begin{align*}
 \sum_{\j \in B_{\i,s,k}} \tmu([\j]) (|\i_{s(i)}|+|\j|-|\i_{s(k)}|) &\ge \frac{1}{C_*} \E_{\i_{s(i)}}(\tau_{s(k)}) \\
&\hspace{-3cm}\ge \frac{1}{C_*} \E_{\i_{s(i)}}(w_{s(k),s(i)})(1- C_*^3P_{s(i),s(k)})-C_*^2P_{s(i),s(k)}L \\
&\hspace{-3cm}\ge \frac{1}{C_*^2} \E_{\tmu} (\tau_{s(k)})(1- C_*^3P_{s(i),s(k)})-C_*^2P_{s(i),s(k)}L\ge  \frac{1-r}{C_*^2} \E_{\tmu} (\tau_{s(k)})-rL.
\end{align*}

 Finally, substituting this into \eqref{hitjump} and applying \eqref{diagonal2}, 
\begin{align*}
\E_\mu(\tau_\delta) &\ge \left(\frac{1-r}{C_*^2} \min_{1 \le m \le M} \E_{\tmu}(\tau_m)-rL\right)(1+1/2+\cdots +1/M) \\
&\ge c \min_{U \in \V_{2\delta/t}} \frac{1}{\mu(U)} \log(1/\delta)=\frac{c}{\delta^{s_f}}\log(1/\delta)
\end{align*}
where in the first line we have used \eqref{ballstocyl2} and Lemma \ref{A_average} and in the second we have used Remark \ref{hitesc}, the fact that $L=O(\log(1/\delta))$, Proposition \ref{prop:u'}(c) to bound $M \ge c\delta^{-\varepsilon}$ as well as Theorem \ref{hit},  \eqref{eq:m_scale}, Proposition \ref{prop:u'}(d)  and the fact that $|g|_\infty=|d\mu/dm|_{\infty}<\infty$ on $\Lambda$.

\end{proof}

\section{Expected cover time: non-uniformly hyperbolic case} \label{section:nonuniform}

In this section we consider some cases where the system $f:I \to I$ does not satisfy our standard assumptions  \ref{a1}-\ref{a6},\eqref{measdiam}-\eqref{eq:ratios}. By considering a first return map (to a subset of $I$) which satisfies  \ref{a1}-\ref{a6}, \eqref{measdiam}-\eqref{eq:ratios}, we will be able to recover results on the expected cover time for the original system.

Let $f:I \to I$ with a conformal measure $m$ and invariant probability measure $\mu$ which is absolutely continuous with respect to $m$. Given an interval $Y \subset I$ with $\mu(Y)>0$ we define $F=f^{\tau_Y}:Y \to Y$ be the first return map to $Y$ and define $\mu_Y=\frac{1}{\mu(Y)}\mu|_Y$.  We let $\Y=\{Y_i\}_i$ be the intervals on which $\tau_Y$ is constant, write $\tau_i=\tau_Y|_{Y_i}$, and assume $F: Y_i\to Y$ is monotone: correspondingly, let $\Y^n$ be the set of $n$-cylinders.

For $x \in Y$ let $R_n(x)=\sum_{i=0}^{n-1}\tau_Y \circ F^i(x)$  denote the $n$th return time of $x$ to $Y$. By Kac theorem, for $\mu_Y$ almost every $x \in Y$,
$$\frac{R_n(x)}n=\frac{1}{n} \sum_{k=0}^{n-1} \tau_Y(F^kx) \xrightarrow{ n \to \infty } \int_Y \tau_Y d\mu_Y=1/\mu(Y).$$

Let $\Z^n$ denote the set of $n$-cylinders for $f$, which are defined just as in \S \ref{section:uniform}. We will require the following property: there exists $N_1 \in \N$ such that
\begin{equation}
\textnormal{(a)} \bigcup_{n=0}^{N_1} f^n(Y)= I ,\; \textnormal{(b) $\Z^{N_1}$ is finite in $Y$}\; \textnormal{and} \; \textnormal{(c)}\sup_{1 \leq n \leq N_1} \sup_{x \in Y}|(f^n)'(x)| < \infty.\label{eq:Ncover}
\end{equation}

As in \S \ref{section:uniform} we begin by stating our result in the special case that  \eqref{eq:m_scale} holds.

\begin{theorem} \label{cover-induced}
Let $f, m$, $\mu$ be as above, in particular we assume  \eqref{eq:m_scale}. Suppose there exists $Y \subset I$ such that $f$ satisfies \eqref{eq:Ncover}. Moreover assume that the first return map $F: Y \to Y$ equipped with the measure $\mu_Y$ satisfies \ref{a1}-\ref{a6}, \eqref{measdiam}-\eqref{eq:m_scale}, with partition $\Y$ replacing $\Z$ there.  Additionally assume that  $\mu_Y(\tau_Y>n) = O(n^{-\gamma})$ for some $\gamma>2$. There exist $0<c<C<\infty$ such that for all $\delta>0$,
$$c\delta^{-s_f} \le \E_\mu(\tau_\delta) \le C\delta^{-s_f}\log(1/\delta).$$ 
Moreover, if $F$ has at least 2 full branches, we have a sharp lower bound
$$c\delta^{-s_f}\log(1/\delta) \le \E_\mu(\tau_\delta) \le C\delta^{-s_f}\log(1/\delta).$$ 
\end{theorem}

Next we state our result for general measures.

\begin{theorem} \label{cover-induced2}
Let $f, m$, $\mu$ be as above. Suppose there exists $Y \subset I$ such that $f$ satisfies \eqref{eq:Ncover}. Moreover assume that the first return map $F: Y \to Y$ equipped with the measure $\mu_Y$ satisfies  \ref{a1}-\ref{a6},\eqref{measdiam}-\eqref{eq:ratios}. Additionally, assume that 
$\mu_Y(\tau_Y>n) = O(n^{-\gamma})$ for some $\gamma>2$.
There exist $0<c<C<\infty$  such that for all $\delta>0$,
$$\frac{c}{M_\mu(\delta)} \le \E_\mu(\tau_\delta) \le \frac{C}{M_\mu(\delta)}\log(1/\delta).$$
In particular if $\md \mu<\infty$ then there exists $\varepsilon>0$ such that
$$c\delta^{-\md \mu+\mathrm{Err}(\delta/\varepsilon)} \le \E_\mu(\tau_\delta) \le C\delta^{-\md \mu -\mathrm{Err}(\varepsilon\delta)}\log(1/\delta).$$
\end{theorem}

We denote
\begin{equation}
\tau^Y_\delta(x):= \inf\{n \ge 1\;:\; \{x, f(x), \ldots, f^n(x)\}\; \textnormal{is $\delta$-dense in $Y$}\}
\end{equation} and
\begin{equation}
\Tau_\delta(x):= \inf\{n \ge 1\;:\; \{x, F(x), \ldots, F^n(x)\}\; \textnormal{is $\delta$-dense in $Y$}\}.
\end{equation}

Fix $\varepsilon>0$ and consider the sets of large deviations
$$A_u=A_{u,\varepsilon}:=\{x \in Y: \exists n \ge u, |R_n(x)/n-1/\mu(Y)|>\varepsilon\}.$$

\begin{lemma}
Under \ref{a6}, $\mu_Y(\tau_Y>n) = O(n^{-\gamma})$ implies $\sum_{u=1}^\infty \mu_Y(A_u)<\infty$.
\label{lem:LD}
\end{lemma}

\begin{proof}
 By \cite[Theorem 2.2]{Gan96}, see also \cite[Theorem 4]{Gan00}, $\psi$-mixing and 
$\mu_Y(\tau_Y>n) = O(n^{-\gamma})$ implies $\mu(A_u) = O(n^{1-\gamma})$.  
\end{proof}

Using Lemma \ref{lem:LD} we will show that under the assumption that $\mu_Y(\tau_Y>n) = O(n^{-\gamma})$ for some $\gamma>2$, it follows that $\E_\mu(\tau_\delta)$ is proportional to $\E_{\mu_Y}(\Tau_\delta)$, and the proof of Theorems~\ref{cover-induced} and \ref{cover-induced2} will follow by applying Theorem \ref{cover} to $F$. 

The following lemma allows us to estimate $\tau_\delta$ from above in terms of $\tau_\delta^Y$, whose expectation can be more easily related to $\E(\Tau_\delta)$. 

\begin{lemma} \label{final-exc}
There exists $\kappa>0$ such that for all sufficiently small $\delta>0$ and all $x \in Y$, $\tau_\delta(x) \le \tau_{\kappa\delta}^Y(x)$.
\end{lemma}

\begin{proof}
By \eqref{eq:Ncover}(c) there exists $C<\infty$ such that $\sup_{1 \leq n \leq N_1}\sup_{x\in Y} |(f^{n})'(x)| \le C$ and we set $\kappa:= 1/(2CN_1(N_1+1))$. Let $P$ be a partition of $Y$ into sets of length at most $\frac{\delta}{CN_1(N_1+1)}$ and at least $\frac{\delta}{2CN_1(N_1+1)}$
with the property that each set in $P$ is contained inside a single element of $\Z^{N_1}$. 

Denote
$$A:= \{f^n(J)\; : \; J \in P, \, 1\le n\le N_1, \, f^n(J) \cap (I\setminus Y) \neq \emptyset\}$$
and note that for each $f^n(J) \in A$, $\diam(f^n(J)) \le C \diam(J) =\frac{\delta}{N_1(N_1+1)} \le \frac{\delta}{2}$. Moreover, by definition of $N_1$,
\begin{equation}
I \setminus Y \subseteq \bigcup_{U \in A} U
\label{Acover}
\end{equation}
 Fix $x \in Y$ and let $k=\tau_{\kappa\delta}^Y(x)$. Note that since each set $J \in P$ has diameter at least $\frac{\delta}{2CN_1(N_1+1)}=\kappa\delta$, this implies that each set in $P$ is visited by the orbit segment $\{x, f(x), \ldots, f^k(x)\}$. Let $\mathcal{J} \subseteq \{1, \ldots, k\}$ be the set of indices $i$ for which $f^i(x) \in Y$. For each $i \in \mathcal{J}$ let $J_i \in P$ denote the partition element that $f^i(x)$ belongs to. Note that each $f^n(J) \in A$ must be visited by the orbit segment $\{x, \ldots, f^k(x)\}$ except possibly the following sets:
$$B:=\{f^j(J_i)\;:\;  i \in \mathcal{J}, \;  1\le j \le N_1 \; \mathnormal{s.t.}\; i+j>k\}.$$
(This is because any set in $B$ might only be visited after time $k$). There are at most $1+2+ \ldots+ N_1=N_1(N_1+1)/2$ distinct sets in $B$. By continuity of $f^{N_1}$ on $\Z^{N_1}$, each set in $B$ is necessarily an interval of length at most $\delta/N_1(N_1+1)$. Therefore, the largest interval contained in $\bigcup_{J' \in B} J'$ has length at most $\delta/2$. In particular by \eqref{Acover} $\{x, \ldots, f^k(x)\}$ must be $\delta$ dense in $I \setminus Y$.
\end{proof}

We are almost ready to estimate $\E_{\mu_Y}(\tau_\delta)$ in terms of $\E_{\mu_Y}(\Tau_\delta)$. We will require the following simple but useful lemma.

\begin{lemma} \label{smalldev}
Suppose $x \in A_u^c$.  Then
$$\tau^Y_\delta(x)>\frac{u}{\mu(Y)} \Rightarrow \Tau_\delta(x)>\frac{u}{1+\varepsilon\mu(Y)}$$
and
$$\Tau_\delta(x)>\frac{u}{1-\varepsilon\mu(Y)} \Rightarrow \tau^Y_\delta(x)>\frac{u}{\mu(Y)}.$$
\end{lemma}

\begin{proof}
Note that for any $x \in Y$, $R_{\Tau_\delta(x)}(x)=\tau_\delta^Y(x)$. 

From our definitions
$$\frac{R_{\Tau_\delta(x)}(x)}{\Tau_\delta(x)}=\frac{\tau^Y_\delta(x)}{\Tau_\delta(x)}=\frac{1}{\mu(Y)}+s$$
where $|s|<\varepsilon$ since $x \in A_u^c$. Therefore if $\tau^Y_\delta(x)>\frac{u}{\mu(Y)}$ then 
$$\Tau_\delta(x)=\frac{\tau^Y_\delta(x)}{1/\mu(Y)+s}>\frac{u/\mu(Y)}{1/\mu(Y)+\varepsilon}=\frac{u}{1+\varepsilon\mu(Y)}$$
which proves the first part. On the other hand if $ \Tau_\delta(x)>\frac{u}{1-\varepsilon\mu(Y)}$ then
$$\tau^Y_\delta(x)=\frac{\Tau_\delta(x)}{\mu(Y)}+s\Tau_\delta(x)>\frac{u}{1-\varepsilon \mu(Y)}(1/\mu(Y)-\varepsilon)=\frac{u}{\mu(Y)}$$
which proves the second part.
\end{proof}
 
The following proposition allows us to estimate $\E_{\mu_Y}(\tau_\delta)$ in terms of $\E_{\mu_Y}(\Tau_\delta)$.

\begin{proposition} \label{muY}
Assume that $\sum_{u=1}^\infty \mu_Y(A_u)<\infty$. There exist constants $c<1<C$ such that for all $\delta>0$,
$$c\E_{\mu_Y}(\Tau_\delta) \le \E_{\mu_Y}(\tau_\delta) \le C\E_{\mu_Y}(\Tau_{\kappa\delta}).$$
\end{proposition}

\begin{proof}
Since $\mu_Y$ is supported on $Y$, by Lemma \ref{final-exc}
$$\E_{\mu_Y}(\tau_\delta^Y) \le \E_{\mu_Y}(\tau_\delta) \le \E_{\mu_Y}(\tau_{\kappa\delta}^Y)$$
so it is enough to obtain lower and upper bounds on $\E_{\mu_Y}(\tau_\delta^Y)$ in terms of $\E_{\mu_Y}(\Tau_\delta)$. To this end
\begin{align*}
\E_{\mu_Y}(\tau_\delta^Y) &= \sum_{n \in \N} \mu_Y(\tau_\delta^Y\ge n) = \sum_{n \in \N} \left(\mu_Y(\tau_\delta^Y \ge n \wedge A_{n\mu(Y)}^c) + \mu_Y(\tau_\delta^Y \ge n \wedge A_{n\mu(Y)}\right) \\
&\le \sum_{n \in \N} \mu_Y\left(\Tau_\delta \ge \frac{n\mu(Y)}{1+\varepsilon \mu(Y)}\right) +\sum_{n \in \N} \mu(A_{n\mu(Y)})  \\
&= \sum_{n \in \N} \mu_Y\left(\frac{\Tau_\delta(1+\varepsilon \mu(Y))}{\mu(Y)} \ge n\right) +\sum_{n \in \N} \mu(A_{n\mu(Y)})  \\
&= \E_{\mu_Y}\left(\frac{\Tau_\delta(1+\varepsilon \mu(Y))}{\mu(Y)}\right) +\sum_{n \in \N} \mu_Y(A_{n\mu(Y)})\\
&= \frac{1+\varepsilon \mu(Y)}{\mu(Y)} \E_{\mu_Y}(\Tau_\delta) + \sum_{n \in \N} \mu_Y(A_{n\mu(Y)})
\end{align*}
where to get the inequality we have used Lemma \ref{smalldev}. Note that since $\sum_{n \in \N} \mu_Y(A_{n\mu(Y)})<\infty$, it is just a constant which is independent of $\delta$, therefore we have the upper bound.

Similarly for the lower bound,
\begin{align*}
\E_{\mu_Y}(\tau_\delta^Y) &\ge  \sum_{n \in \N} \mu_Y(\tau_\delta^Y \ge n \wedge A_{n\mu(Y)}^c) \ge \sum_{n \in \N}  \mu_Y\left(\Tau_\delta \ge \frac{n\mu(Y)}{1-\varepsilon \mu(Y)}\right)\\
&= \E_{\mu_Y}\left(\frac{\Tau_\delta(1-\varepsilon \mu(Y))}{\mu(Y)}\right) = \left(\frac{1-\varepsilon\mu(Y)}{\mu(Y)}\right)\E_{\mu_Y}(\Tau_\delta)
\end{align*}
where again we have used Lemma \ref{smalldev}.
\end{proof}

Proposition \ref{muY} will allow us to obtain a lower bound on $\E_\mu(\tau_\delta)$ since $\E_\mu(\tau_\delta) \ge \mu(Y)\E_{\mu_Y}(\tau_\delta)$. Next we will obtain an upper bound on $\E_\mu(\tau_\delta)$ in terms of $\E_{\mu_Y}(\tau_\delta^Y)$. 

For this we will view the system as a tower $(\Delta, f_\Delta, \mu_\Delta)$ built over $(Y, F = f^{\tau_Y}, \mu_Y)$, recalling we denote the domains of $F$ by $\Y=\{Y_i\}_i$ and $\tau_Y|_{Y_i} = \tau_i$.  We will borrow the language of Young Towers, but we do not assume any structure other than that assumed in Theorems~\ref{cover-induced} and \ref{cover-induced2}, for example we do not assume the return map to the base is Markov. We call the base, which corresponds to $Y$, $\Delta_0$, and we'll use notation $\Delta_{i, j}$ for domains in $\Delta$, where $\Delta_{i, 0}\subset \Delta_0$, corresponding to $Y_i$; and $f_\Delta^j(\Delta_{i, 0})=\Delta_{i,j}$ for $j\le \tau_i-1$ and $f_\Delta^{\tau_i}(\Delta_{i, 0})\subset \Delta_{0}$ (corresponding to $F(Y_i)\subset Y$). Points in $\Delta_{i, j}$ are denoted $(x, j)$ for $x\in Y$.  Then $f_\Delta^{\tau_Y}$ on $\Delta_0$ corresponds to $F$ on $Y$.   We let $\pi:\Delta \to I$ be the natural projection $\pi(x,j)=f^j(x)$. Note that since $F$ is a first return map, $\pi|_{\Delta_0}:\Delta_0\to Y$ is bijective. We define $\mu_\Delta$ to be the measure $\mu_\Delta(\Delta_{i,j})=\mu_\Delta(\Delta_{i,0})=\mu(Y_i)$. Then $\pi_*\mu_\Delta=\mu$ and $f\circ \pi=\pi \circ f_\Delta$. Let $\mu_{\Delta_0} = \mu_{\Delta}|_{\Delta_0}$ - we will similarly restrict $\mu_\Delta$ to $\Delta_{i, j}$ to get $\mu_{\Delta_{i, j}}$.  We will also consider the symbolic coding for $F$ as in \S \ref{section:uniform}; in this setup $\Sigma$ can be taken to be the set of all $\i=(i_0, i_1 \ldots)$ such that for some $x \in Y$, $x \in \Delta_{i_0,0}$, $F(x) \in \Delta_{i_1,0}$ and so on. Moreover, in this case $\Pi(\i)=x$. Let $\tmu_Y$ be the measure on $\Sigma$ such that $\Pi_*\tmu_Y=\mu_Y$, which by our assumptions is quasi-Bernoulli. The relationship between $\tmu_Y$ and $\mu_{\Delta_0}$ is the following: $\mu_{\Delta_0}=\mu(Y)\Pi_*\tmu_Y$. 

We can also extend the cover time function to $\Delta_0$ in the natural way:
$$\tau_\delta^{\Delta_0}(x,j):=\inf\{n \in \N: \{\pi(x,j), \ldots, \pi(f_\Delta^n(x,j))\}\; \textnormal{is $\delta$ dense in $\Delta_0$}\},$$
where the metric on $\Delta_0$ is the one induced from $Y$. Since $f\circ \pi=\pi \circ f_\Delta$ and by definition of $\mu_\Delta$, we have $\E_\mu(\tau^Y_{\kappa\delta})=\E_{\mu_\Delta}(\tau_{\kappa\delta}^{\Delta_0}).$

\begin{proposition} \label{tower}
Suppose $F$ and $\mu_Y$, with partition $\Y$, satisfy \ref{a1}-\ref{a4}, \ref{a6}, \eqref{measdiam} and \eqref{bd} and that  $\mu_Y(\tau_Y>n) = O(n^{-\gamma})$ where $\gamma>2$. Then for some $\varepsilon>0$, and $C>0$
$$\E_\mu(\tau_\delta) \le C \E_{\mu_Y}(\tau_{\varepsilon\delta}^Y).$$
\end{proposition}

\begin{proof}
By Lemma \ref{final-exc} we have
$$\E_\mu(\tau_\delta) \le \E_\mu(\tau^Y_{\kappa\delta})=\E_{\mu_\Delta}(\tau_{\kappa\delta}^{\Delta_0}).$$
For any $x \in \Delta_{i,0}$,
$$ \tau^{\Delta_0}_{\kappa\delta}((x, j)) \le \tau^{\Delta_0}_{\kappa\delta}(F(x),0)+\tau_i-j.$$
Hence
$$\E_{\mu_{\Delta_{i, j}}}\left(\tau_{\kappa\delta}^{\Delta_0}\right)\le \E_{\mu_{\Delta_{i, 0}}}\left(\tau_{\kappa\delta}^{\Delta_0}\circ F+\tau_i-j \right).$$
Summing over the column (i.e. summing over $j$) we obtain
$$\sum_{j=0}^{\tau_i-1}\E_{\mu_{\Delta_{i, j}}}\left(\tau_{\kappa\delta}^{\Delta_0}\right) \le \tau_i\E_{\mu_{\Delta_{i, 0}}}\left(\tau_{\kappa\delta}^{\Delta_0}\circ F \right)+  \mu_{\Delta}(\Delta_{i, 0})\sum_{j=0}^{\tau_i-1}( \tau_i-j).$$
Now summing over $i$ we obtain
\begin{equation}\E_{\mu_\Delta}(\tau^{\Delta_0}_{\kappa\delta})=\sum_i\sum_{j=0}^{\tau_i-1} \E_{\mu_{\Delta_{i,j}}}(\tau^{\Delta_0}_{\kappa\delta}) \le \sum_i \tau_i\E_{\mu_{\Delta_{i,0}}}(\tau^{\Delta_0}_{\kappa\delta} \circ F) +R \label{tower exp}
\end{equation}
where $R:=\sum_i\mu_{\Delta}(\Delta_{i, 0})\sum_{j=0}^{\tau_i-1}j= \sum_nn\mu_Y(\tau_Y>n)<\infty$ since we have assumed $\mu_Y(\tau_Y>n) = O(n^{-\gamma})$ for some $\gamma>2$. 

Now, let $\mathcal{P}_\delta$ be defined as in \eqref{Pd} for the system $F:Y \to Y$ (which is possible by our assumptions on $F$ and $\mu_Y$). Put $\mathcal{P}=\mathcal{P}_{\kappa\delta/T}$. By definition of $\mu_{\Delta_{i,0}}$,
\begin{align*}
\E_{\mu_{\Delta_{i,0}}}(\tau^{\Delta_0}_{\kappa\delta} \circ F) &=\mu(Y) \E_{\mu_Y|\Pi([i])}(\tau_{\kappa\delta}^Y \circ F) \le \mu(Y) \E_{\tmu_Y|[i]}(\tau_{\mathcal{P}}\circ F) = \mu(Y) \sum_{\j \in \P} \tmu_Y([i\j])\left(\sum_k \tau_{j_k}\right) \\
&\le C_*  \mu(Y) \tmu_Y([i]) \sum_{\j \in \P} \tmu_Y([\j])\left(\sum_k \tau_{j_k}\right)= C_*\mu_{\Delta_0}(\Delta_{i,0}) \E_{\tmu_Y}(\tau_{\P}) \\
&\le C_*\mu_{\Delta_0}(\Delta_{i,0}) \E_{\mu_{Y}}(\tau_{t\kappa\delta/3T}^{Y})= \frac{C_*}{\mu(Y)}\mu_{\Delta_0}(\Delta_{i,0}) \E_{\mu_{\Delta_0}}(\tau_{t\kappa\delta/3T}^{\Delta_0})
\end{align*}
where the first inequality follows by \eqref{ballstocyl}, the second inequality by the quasi-Bernoulli property \ref{a5} of $\mu_Y$ and the final inequality follows by \eqref{ballstocyl2}. 

Now putting back into \eqref{tower exp} and using that  $\sum_i \tau_i\mu_{\Delta_0}(\Delta_{i,0})=1$ we obtain
\begin{align*}
\E_\mu(\tau_\delta) \le \E_{\mu_\Delta}(\tau_{\kappa\delta}^{\Delta_0}) &\le \frac{C_*}{\mu(Y)}\sum_i \tau_i \mu_{\Delta_0}(\Delta_{i,0}) \E_{\mu_{\Delta_0}}(\tau_{t\kappa\delta/3T}^{\Delta_0})+R  \\
&=\frac{C_*}{\mu(Y)} \E_{\mu_{\Delta_0}}(\tau_{t\delta/12T}^{\Delta_0})+K=C_* \E_{\mu_Y} (\tau_{t\kappa\delta/3T}^Y)+R
\end{align*}
which completes the proof.
\end{proof}

\begin{proof}[Proof of Theorems \ref{cover-induced} and \ref{cover-induced2}]
By our assumption, there exists $Y \subset I$ such that the first return map $F:Y \to Y$ equipped with the induced measure $\mu_Y$ satisfies the assumptions of Theorem \ref{cover}. Since Lemma~\ref{lem:LD} implies $\sum_{u=1}^\infty \mu_Y(A_u)<\infty$,
\begin{align*}
\E_\mu(\tau_\delta) \ge \mu(Y) \E_{\mu_Y}(\tau_\delta) \ge c\mu(Y) \E_{\mu_Y}(\Tau_\delta) 
\end{align*}
where in the second inequality we have used Proposition \ref{muY}. Now, in the setting for general $\mu$ (ie. proof of Theorem \ref{cover-induced2}) the lower bound is obtained by applying Theorem \ref{cover} to deduce that $ \E_{\mu_Y}(\Tau_\delta) \ge \frac{c}{M_\mu(\delta/\varepsilon)}$. In the setting where  \eqref{eq:m_scale} holds (ie. proof of Theorem \ref{cover-induced}) observe that  \eqref{eq:m_scale} holds on $Y$, thus we can again apply Theorem \ref{cover} to deduce that $ \E_{\mu_Y}(\Tau_\delta) \ge c \delta^{-s_f}\log(1/\delta)$.

For the upper bound we have
\begin{align*}
\E_\mu(\tau_\delta) \le C \E_{\mu_Y}(\tau_\delta^Y) \le C\E_{\mu_Y}(\tau_\delta) \le C\E_{\mu_Y}(\Tau_{\kappa\delta}) 
\end{align*}
where in the first inequality we have used Proposition \ref{tower}, in the third we have used Proposition \ref{muY}. In the case for general $\mu$  (ie. proof of Theorem \ref{cover-induced2}) we obtain the upper bound by applying Theorem \ref{cover} to get $\E_{\mu_Y}(\Tau_{\kappa\delta})  \le  \frac{C}{M_\mu(\varepsilon\kappa \delta)}\log(1/\delta)$. In the setting where  \eqref{eq:m_scale} holds we can again apply Theorem \ref{cover} to deduce that $\E_{\mu_Y}(\Tau_{\kappa\delta}) \le C\delta^{-s}\log(1/\delta)$.
\end{proof}

\section{Examples} \label{examples}

\subsection{Full branched interval maps with acip}

Let $f$ be a full branched map of an interval $I$ (i.e. $\Lambda=I$) and assume each of the branches $f|_{Z}$ is $C^{1+Lip}$, such that there is a uniform bound on the Lipschitz constants of the derivatives. Assume $|Df| \ge \gamma^{-1}>1$ on each $Z \in \Z$ and $\frac{|Z|}{|Z'|}$ is uniformly bounded for adjacent intervals $Z, Z'\in \Z$. The potential $\phi=-\log|Df|$ gives rise to an acip $\mu$ for $f$. 

Since $f$ is full-branched, it is Markov and satisfies BIP. The potential $\phi$ is Lipschitz with uniform Lipschitz constants, hence satisfies \eqref{bd}. Moreover $\mu$ is necessarily Gibbs for the potential $\phi=-\log|Df|$, so by the discussion in \S\ref{GM} the system $(f,\mu)$ satisfies \ref{a1}-\ref{a6}. In this case the conformal measure $m$ is Lebesgue measure, hence \eqref{measdiam} is satisfied for $s=1$.  For \eqref{eq:ratios}, this follows from the uniform bounds on the density and the bound on $\frac{|Z|}{|Z'|}$.  In particular Theorem \ref{cover} applies.


\subsection{Gauss map}
The Gauss map $f:(0, 1]\to (0, 1]$ is defined by $f(x)=\frac{1}{x} \mod 1$, so defining our partition $\Z$ as $\{Z_n\}_n=\{(\frac{1}{n+1}, \frac{1}{n})\}_n$, $f$ is continuously differentiable on each element of $\Z$.  We set   $\phi=-\log|Df|$ and $m$ be Lebesgue measure on $(0,1]$. Then we have an invariant density is $g(x)=\frac{1}{\log 2} \frac{1}{1+x}$ defining an acip $\mu$: call $((0, 1], f,\phi,  \mu)$ the \emph{Gauss system}.

\ref{a1} fails for this system, however as in \cite[Section 2.6.3]{BruDemTod18} the potential satisfies the following weaker H\"older distortion control on the cylinder sets and the system can be shown to satisfy our theory as below.

\begin{lemma} \label{holder}
There exists $C_d>0$ such that $|e^{S_n\phi(x)-S_n\phi(y)}-1| \le C_d|f^nx-f^ny|^{1/2}$ whenever $f^i(x)$ and $f^i(y)$ lie in the same element of $\Z$ for each $i=0, \ldots, n-1$. 
\end{lemma}
\begin{proof}
See \cite[Lemma 2.9]{BruDemTod18}.
\end{proof}

We'll show that Lemma \ref{holder} suffices to prove Proposition \ref{prop:u} and Proposition \ref{ly}, so that Theorem \ref{cover} holds for the Gauss system. 

First we verify that the remainder of our assumptions hold for the Gauss system. For \ref{a2} notice that $|e^{\phi}|_\infty \le 1$ while $|e^{S_2\phi}|_\infty <1$ thus $|e^{S_n\phi}|_\infty$ converges to 0 exponentially fast in $n$, while $\Lp^n1$ converges to $g$ which is bounded away from 0 on $[0,1]$. \ref{a3} holds since $f$ is full branched and the potential satisfies Lemma \ref{holder}. \ref{a4} holds since $f$ is full branched. \ref{a5} and \ref{a6} hold since $\mu$ is Gibbs. \eqref{measdiam} and \eqref{eq:m_scale} hold for $s=s_f=1$ since $m$ is Lebesgue measure. \eqref{bd} follows from Lemma \ref{holder}. \eqref{eq:ratios} holds since $1 \le \frac{\mu(Z_{n-1})}{\mu(Z_n)} \le \frac{n+1}{n-1}$ for all $n \ge 2$. 
Instead of \eqref{n1} we choose $n_1$ such that $\frac{2n_1+5}{2} \sup_I e^{S_{n_1} \phi} |\phi|_{BV}<1$ (the reason for this will be made clear later).

We note that Proposition \ref{prop:u} still holds for the Gauss system since the only time \ref{a1} is used in the proof is in \eqref{a1used}, where \eqref{bd} can be used instead, to obtain the same conclusion.

We also verify that Lemma \ref{prels} holds: (a) holds by induction on \eqref{eq:a2'}, (b) holds with $C_d=1$ since $e^\phi$ is monotonic on each $Z \in \Z$ so $|\phi|_{BV,Z} \le \sup_Z e^\phi$ and (c) holds by induction on (b), using (a). 

Finally, we are ready to demonstrate that the uniform Lasota-Yorke inequalities hold in Proposition \ref{ly}. \ref{a1} is used in two places in the proof: in \eqref{a1used2} and \eqref{a1used3}. \eqref{a1used2} holds with $C_d=1$ by Lemma \ref{prels}(b).  \eqref{a1used3} holds using \eqref{bd}. Hence the contracting term in \eqref{ly3} becomes $\frac{2n+5}{2} \sup_I e^{S_{n_1} \phi} |\phi|_{BV}$, which is the same expression, just with $C_d=1$, hence since $n_1$ was chosen such that $\frac{2n_1+5}{2} \sup_I e^{S_{n_1} \phi} |\phi|_{BV}<1$, the proof of Proposition \ref{ly} is complete. Hence we have shown that Theorem \ref{cover} holds for the Gauss system.

\subsection{Manneville-Pomeau maps}\label{sec:MP}

For $\alpha \in (0,1)$, we will study the class of Manneville-Pomeau maps defined by
\begin{equation*}
f=f_\alpha:x\mapsto \begin{cases} x(1+2^\alpha x^\alpha) & \text{ if } x\in [0, 1/2),\\
2x-1 & \text{ if } x\in [1/2, 1].\end{cases}
\end{equation*}
(This is the simpler form given by Liverani, Saussol and Vaienti, often referred to as LSV maps.) We let $\Z=\{[0, 1/2), [1/2, 1]\}$.  These maps all have an acip $\mu=\mu_\alpha$, which is an equilibrium state for the potential $\phi=-\log Df$.  Our first returns will be taken to the set $Y=[1/2, 1]$. 
The induced system $(Y, F=f^{\tau_Y})$ is a full-branched Gibbs Markov map with respect to the induced potential $-\log DF$: letting $Y_i$ denote the interval on which $\tau_Y=i$, $F|_{Y_i}= f^i|_{Y_i}$ is a diffeomorphism.  Since $Y \cup f(Y)=[0,1]$, (\ref{eq:Ncover}) holds with $N_1=1$.  We write $\Y=\{Y_i\}_i$. Adjacent intervals, $Y_i$ and $Y_{i-1}$ can be shown (see for example \cite[Corollary 1]{Sar02}) to have $\frac{\mu_Y(Y_{i-1})}{\mu_Y(Y_{i})} \sim \left(\frac{i}{i-1}\right)^{\frac1\alpha+1}$, so \eqref{eq:ratios} holds.  Moreover, the conformal measure is Lebesgue, so \eqref{eq:m_scale} holds with $s_f=1$, and the induced system is $\psi$-mixing, which we used in the proof of Proposition~\ref{prop:RTS} (though in fact we only require $\phi$-mixing, which is a weaker condition).  Thus to apply Theorem~\ref{cover-induced}, it is sufficient that $\mu(\tau_Y>n) = O(n^{-\gamma})$ for some $\gamma>2$.  It is well known that for $f_\alpha$, $\mu(\tau_Y>n) = O(n^{-1/\alpha})$ (again, see for example \cite[Corollary 1]{Sar02}), so our results apply to the case $\alpha\in (0, 1/2)$.

\begin{remark}
Hitting Time statistics results hold for all $\alpha\in (0, 1)$ at all points, see \cite{FreFreTodVai16}, so it is a natural question to ask if the results here should also hold in that generality.  However, in this setting we are required to take expectations, which means summing over more quantities, so it is not clear if the restriction to $(0, 1/2)$, which is used twice in our argument, is an artefact of our proof method or is intrinsic for this problem.  It may be relevant to note that the range $\alpha\in (0, 1/2)$ is the range of parameters for which the Central Limit Theorem holds: for $\alpha\in (1/2, 1)$ it fails (this is also seen in the i.i.d. case with observables outside $L^2$).  
\end{remark}


\subsection{Systems with slow covering}

In previous related studies of the cover time \cite{JurMor20, bjk}, the leading term for the expected cover time $\E_\mu(\tau_\delta)$ always obeyed a power law in $1/\delta$, that is, $\md \mu<\infty$ and $\E_\mu(\tau_\delta) \approx \left(1/\delta\right)^{\md \mu}$. However, our setup presents natural examples of systems in which average orbits become dense in the state space at a much slower rate. In these examples $\md \mu=\infty$, therefore the formula for the expected cover time can no longer be expressed as a power law in $1/\delta$. 

\begin{example} Let $f:[0,1] \to [0,1]$ be the Manneville-Pomeau map
\begin{align*}
f(x)= 
\begin{cases} f_1(x)= x(1+2^\alpha x^\alpha), \hspace{1.5cm} x\in [0,1/2) \\
f_2(x)= 2x-1, \hspace{2.4cm} x \in [1/2,1]
\end{cases}
\end{align*}
although this time we equip it with the measure $\mu$ which is given as the pushforward of the $(1/2,1/2)$ Bernoulli measure on $\{1,2\}^\N$ through the coding map $\Pi:\{1,2\}^\N \to [0,1]$,
$$\Pi(i_0 \ldots i_{n-1}\ldots )=\lim_{n \to \infty} f_{i_0}^{-1} \circ \cdots \circ f_{i_{n-1}}^{-1}([0,1)).$$
Another way of viewing this is that $\mu$ is the measure of maximal entropy, the equilibrium state for the constant potential $-\log 2$.  

Exactly as in \S \ref{sec:MP}, we write $\Z=\{[0, 1/2), [1/2, 1]\}$ and we can take a first return map $F$ to $Y=[1/2, 1]$ with domains $\Y=\{Y_i\}_i$.  The induced potential now is $\phi(x) = -\tau_Y(x)\log 2$, we see that this system satisfies the conditions in Section~\ref{props}.  Moreover, since we can compute $\mu_Y(Y_i) = 1/2^i$ (note that the conformal measure and the invariant measure coincide here), \eqref{eq:ratios} clearly holds with $C=2$ and \eqref{measdiam} holds since the diameters of $Y_i$ are polynomial in $k$, while the $\mu$-measures are exponential.  Theorem \ref{cover-induced2} therefore applies.

Then for each $n \in \N$, $\mu(\Pi([1^n]))=1/2^n$ whereas $\diam(\Pi([1^n])) \sim \frac{1}{2} (\alpha n)^{-1/\alpha}$. In particular $M_\mu(\delta) \le  (2^{1/\delta^\alpha})^{-1/(\alpha 2^\alpha) }$ so by Theorem \ref{cover-induced2}, $\E_\mu(\tau_\delta) \ge  (2^{1/\delta^\alpha})^{1/(\alpha 2^\alpha) }$. The key point here is that it is very difficult for this system to cover a neighbourhood of 0, which drives the expected cover time up.

\end{example}

\begin{example} Let $a_0=0$ and for $k\ge 1$, let $a_k = \sum_{j=1}^k\frac1{ck^2}$ where $c=\frac{\pi^2}6$.   Then let $f:[0,1) \to [0,1)$ be the map
$$f(x) = ck^2(x- a_{k-1}) \text{ for } x\in [a_{k-1}, a_k):=Z_k$$
equipped with the measure $\mu$ which is the equilibrium state (and conformal measure) for the potential $\phi|_{Z_k} =-k\log2$, for which the conditions in Section~\ref{props} hold.  Moreover, for $\Z=\{Z_k\}_k$,
\eqref{eq:ratios} holds with $C=2$ and \eqref{measdiam} holds since the diameters of $Z_k$ are polynomial in $k$, while the $\mu$-measures are exponential.

This is a Gibbs-Markov system so Theorem \ref{cover2} applies. Then for each $k \in \N$, $\mu(Z_k)=\frac{1}{2^k}$ whereas $\diam(Z_k)=\frac{1}{ck^2}$. In particular for $\delta = \frac{1}{ck^2}$, $M_\mu(\delta) \le \frac{1}{2^{\sqrt{1/c\delta}}}$ so by Theorem \ref{cover2}, $\E_\mu(\tau_\delta) \gtrsim 2^{\sqrt{1/c\delta}}$. In this example, it is very difficult for the system to cover a neighbourhood of 1, which drives the expected cover time up.
\end{example}

\appendix

\section{Proof of Propositions \ref{prop:u} and \ref{prop:u'}}

In this section we prove Propositions \ref{prop:u} and \ref{prop:u'}, and we begin with the former. We start with the simpler case where $\Lambda =I$. Following this, we will describe how the proof can be adapted to hold for more general $\Lambda$. 

\subsection{Proof of Proposition \ref{prop:u} for interval maps} 

The proof in the case that $f$ has finitely many branches is trivial, therefore we focus our attention on the case where $f$ has infinitely many branches.

Throughout the proof, given an interval $J$ we'll denote its diameter by $|J|$. The rough idea is the following. For each small $r$ we will partition $I$ into intervals of diameter roughly $r$ in an iterative way. If a union of neighbouring 1-cylinders can be taken of the correct length then this union forms a ball in $\U_\delta$. For all 1-cylinders which are too large, we split it into 2-cylinders and consider unions of these which have both the correct diameter and whose diameter comprises a small proportion of the diameter of the 1-cylinder. On the other hand, for all 2-cylinders which are too large we iterate the process by splitting them into 3-cylinders. This is continued until we reach $n_1$-cylinders, at which point any $n_1$-cylinder which is too large is split into arbitrary intervals (not necessarily unions of $(n_1+1)$-cylinders) which have the correct diameter and whose diameter comprises a small proportion of the $n_1$-cylinder.

For any interval $J$ let $\ell(J)$ and $r(J)$ denote the left and right end points of $J$ respectively. Beginning from $\ell(I)$ we can uniquely choose a (possibly infinite) union of 1-cylinders $\cup_{J_1 \in \alpha_1} J_1$ which forms an interval such that $\ell(\cup_{J_1 \in \alpha_1} J_1)=\ell(I)$, $|\cup_{J_1 \in \alpha_1} J_1| \ge \delta/2$ and such that for any subcollection $\alpha_1'\subset \alpha_1$ for which $\cup_{J_1 \in \alpha_1'}J_1$ is an interval starting at $\ell(I)$, we have $|\cup_{J_1 \in \alpha_1'}J_1| <\delta/2$. 
If $|\cup_{J_1 \in \alpha_1} J_1| < 3\delta/2$ then we add $\cup_{J_1 \in \alpha_1} J_1$ to $U_1$. Otherwise, there exists $J_1^* \in \alpha_1$ such that $r(J_1^*)=r(\cup_{J_1 \in \alpha_1} J_1)$ and $|\cup_{J_1 \in \alpha_1} J_1|-|J_1^*|<\delta/2$, in particular $|J_1^*| \ge \delta$. In this case we add $\cup_{J_1 \in \alpha_1} J_1$ to $V_1$.  

We then repeat the above, replacing $\ell(I)$ by $r(\cup_{J_1 \in \alpha_1} J_1)$, and continue the process until either $U_1\cup V_1$ cover $I$ or we are left with one interval with right hand end point $r(I)$ and diameter less than $\delta/2$. In this case we add this interval to $R_1$. 

Recall that $c_m$ is given by \ref{a4} and $C_m$ is given by \eqref{measdiam}. Now, let $0<\tilde{\beta}<c_m$ and put 
\begin{equation*}\beta=\frac{\tilde{\beta}c_m}{C_m C_{bd}^s(n_1+1)(1+C_d)^3}.\end{equation*}
Next, for each $A \in V_1$ we take the final 1-cylinder $J_1^*$ contained in $A$ and beginning from $\ell(J_1^*)$ we can uniquely choose a union of 2-cylinders $\cup_{J_2 \in \alpha_2}J_2$ which forms an interval such that $\ell(\cup_{J_2 \in \alpha_2} J_2)=\ell(J_1^*)$, $|\cup_{J_2 \in \alpha_2} J_2| \ge \beta \delta/4$ and such that for any subcollection $\alpha_2'\subset \alpha_2$ for which $\cup_{J_2 \in \alpha_2'}J_2$ is an interval starting at $\ell(J_1^*)$, we have $|\cup_{J_2 \in \alpha_2'}J_2| < \beta \delta/4$. If $|\cup_{J_2 \in \alpha_2}J_2|<\beta \delta/2$ then we add $\cup_{J_2 \in \alpha_2}J_2$ to $U_2$.  Otherwise, there exists $J_2^* \in \alpha_2$ such that $r(J_2^*)=r(\cup_{J_2 \in \alpha_2} J_2)$ and $|\cup_{J_2 \in \alpha_2} J_2|-|J_2^*|<\beta \delta/4$, in particular $|J_2^*| \ge \beta\delta/4$. In this case we add $\cup_{J_2 \in \alpha_2} J_2$ to $V_2$. We then repeat the above, replacing $\ell(J_1^*)$ by $r(\cup_{J_2 \in \alpha_2} J_2)$, and continue the process until either $U_1\cup U_2\cup V_2 \cup R_1$ cover $I$ or we are left with a finite number of intervals (all of which appear on the left or right hand sides of intervals in $V_1$). The ones on the left will have length at most $\delta/2$ and will be made up of a union of 1-cylinders, and we put these into $L_1$. The ones on the right will have length at most $\beta \delta/4$ and these will be put into $R_2$.

We can continue the above process inductively for each $2 \le n \le n_1-1$. For each $A \in V_n$ we take the final $n$-cylinder $J_n^*$ contained in $A$ and beginning from $\ell(J_n^*)$ we can uniquely choose a union of $n+1$-cylinders $\cup_{J_{n+1} \in \alpha_{n+1}}J_{n+1}$ which forms an interval such that $\ell(\cup_{J_{n+1} \in \alpha_{n+1}} J_{n+1})=\ell(J_n^*)$, $|\cup_{J_{n+1} \in \alpha_{n+1}} J_{n+1}| \ge \delta (\beta/4)^n$ and such that for any subcollection $\alpha_{n+1}'\subset \alpha_{n+1}$ for which $\cup_{J_{n+1} \in \alpha_{n+1}'}J_{n+1}$ is an interval starting at $\ell(J_n^*)$, we have $|\cup_{J_{n+1} \in \alpha_{n+1}'}J_{n+1}| < \delta(\beta/4)^{n}$. If $|\cup_{J_{n+1} \in \alpha_{n+1}}J_{n+1}|<\delta (\beta/2) (\beta/4)^{n-1}$ then we add $\cup_{J_{n+1} \in \alpha_{n+1}}J_{n+1}$ to $U_{n+1}$. Otherwise, there exists $J_{n+1}^* \in \alpha_{n+1}$ such that $\delta(J_{n+1}^*)=r(\cup_{J_{n+1} \in \alpha_{n+1}} J_{n+1})$ and $|\cup_{J_{n+1} \in \alpha_{n+1}} J_{n+1}|-|J_{n+1}^*|<\delta(\beta/4)^n$, in particular $|J_{n+1}^*| \ge \delta(\beta/4)^n$. In this case we add $\cup_{J_{n+1} \in \alpha_{n+1}} J_{n+1}$ to $V_{n+1}$. We then repeat the above, replacing $\ell(J_n^*)$ by $r(\cup_{J_{n+1} \in \alpha_{n+1}} J_{n+1})$, and continue the process until either $U_1\cup \cdots \cup U_{n+1}\cup V_{n+1} \cup L_1\cup \cdots \cup L_{n-1}\cup R_1 \cup \cdots \cup R_n$ cover $I$ or we are left with a finite number of intervals (all of which appear on the left or right hand sides of intervals in $V_{n+1}$). The ones on the left will have length at most $\delta(\beta/4)^{n-1}$ and will be made up of a union of $n$-cylinders, and we add this to $L_n$. The ones on the right will have length at most $\delta (\beta/4)^n$, which we put into $R_{n+1}$. 

Let $L_\delta$ be the smallest integer for which any $n_1+L_\delta$ cylinder has diameter at most $\frac{\delta}{3}(\frac{\beta}{4})^{n_1}$. Since $f$ is uniformly expanding, $L_\delta=O(\log(1/\delta))$. For each $A \in V_{n_1}$ we take the final $n_1$-cylinder $J_{n_1}^*$ contained in $A$ and split it up into intervals of length at least $(\beta/4)^{n_1}\frac{\delta}{3}$ and at most $(\beta/4)^{n_1}\delta$, in such a way that each interval is determined as a union of $n_1+L_\delta$ cylinders.
We call this collection $U_{n_1}$.

Append each interval in the set $\bigcup_{i=1}^{n_1} L_i$ to the nearest interval to its right belonging to $\bigcup_{i=1}^{n_1}U_i$, and name the new collection $\bigcup_{i=1}^{n_1}U_i'$.  Then, append each interval in the set $\bigcup_{i=1}^{n_1} R_i$ to the nearest interval to its left belonging to $\bigcup_{i=1}^{n_1}U_i'$. We define this new collection to be $\U_{\delta}$.  \ref{Ub} and \ref{Uc} follow from construction. \ref{Ue} holds since each $U \in\U_\delta$ is given by a finite or countable collection of $n_1+L_\delta$ cylinders. To see \ref{Ud}, note that for any $U \in \U_\delta$ and any $Z \in \Z^i$ where $U \cap Z \neq \emptyset$, if $Z$ is not a subset of $U$ then $|Z| \ge \delta(\beta/4)^{i-1}$ and $|U\cap Z| \le\delta(\beta/4)^{i}+\delta(\beta/2)(\beta/4)^{i-1}$ hence
\begin{align}
\frac{m(U \cap Z)}{m(Z)} & \le \frac{\sup_Z e^{S_i \phi}}{\inf_Z e^{S_i \phi}} \frac{m(f^i(U \cap Z))}{m(f^i(Z))} \le \frac{C_m(C_d+1)}{c_m}\diam(f^i(U \cap Z))^s \label{a1used} \\
&\le \frac{C_m (C_d+1)C_{bd}^s}{c_m} \left(\frac{\diam(U \cap Z)}{\diam(Z)}\right)^s \le \frac{ C_m(C_d+1)C_{bd}^s}{c_m}\left(\frac{ \delta(3\beta/4)(\beta/4)^{i-1}}{\delta(\beta/4)^{i-1}} \right)^s \nonumber\\
&<\frac{C_m(C_d+1)C_{bd}^s\beta}{c_m}=\frac{\tilde{\beta}}{(n_1+1)(1+C_d)^2} \nonumber\end{align} 
where the first inequality is conformality, the second is by \eqref{measdiam}, \ref{a1} and \eqref{a4i}, the third is by \eqref{bd} and the fact that $\diam(f^n(Z)) \le 1$. Finally, to see \ref{Ua}, notice that since $f$ is an interval map, by \ref{a3} $\supp\, m=I$ and so it is enough to show that the diameter of any set in $\U_\delta$ can be bounded above and below by a constant times $\delta$, where the constants are independent of $\delta$. The maximum length any interval in $U_\delta$ can be is $\max\{3\delta/2+\delta/2, \delta/2+\delta \beta/4+\cdots +\delta(\beta/4)^{n_1-1}+(\beta/4)^{n_1}\delta\}$. The minimum length is $(\beta/4)^{n_1}\frac{\delta}{3}$. This proves \ref{Ua}.

{\bf Proof of Proposition \ref{prop:u} for maps $f:\Lambda \to \Lambda$}

The proof of Proposition \ref{prop:u} for more general maps $f:\Lambda \to \Lambda$ (where $\Lambda$ is not an interval) follows similarly to above, with some small changes, which we discuss here. Let $I$ be the closed interval beginning at $\inf_{x \in Z}\inf_{Z \in \Z} x$ and ending at $\sup_{x \in Z}\sup_{Z \in \Z} x$.

In the interval map setting we had that $\bigcup_{Z \in \Z^{n+1}} Z=\bigcup_{Z' \in \Z^n} Z'$, whereas more generally we only have $\bigcup_{Z \in \Z^{n+1}} Z\subset \bigcup_{Z' \in \Z^n} Z'$. We $\H^n$ be the set of intervals for which
$$\bigcup_{H \in \H^n} H \cup \bigcup_{Z \in \Z^{n}} Z=  \bigcup_{Z' \in \Z^{n-1}} Z' .$$
Note that for any $H \in \H^n$, $H \cap \Lambda=\emptyset$. Now we broadly follow the method from the previous section replacing $\Z^n$ by $\Z^n \cup \H^n$ for each level $1 \le n \le n_1$. Namely, at the first step instead of taking minimal unions of intervals from $\Z^1$ whose length exceeds $\delta/2$, we take minimal unions of intervals from $\Z^1 \cup \H^1$ whose length exceeds $\delta/2$. If the length of the union is less than $3\delta/2$ we add this interval to $U_1$ as before. If it exceeds $3\delta/2$ but the last interval added was an interval $H$ from $\H^1$, we simply take a smaller proportion of the interval $H$ so that the union has total length less than $3\delta/2$ and add this interval to $U_1$. Finally, analogously to before if the length of the union exceeds $3\delta/2$ and the last interval added was an interval $Z$ from $\Z^1$, then it is added to $V_1$. 

Let $$\beta=\frac{\tilde{\beta}c_m}{2C_m C_{bd}^s(n_1+1)(1+C_d)^3}.$$
At the next stage, for each set in $V^1$, we take the rightmost cylinder from $\Z^1$ contained in the set and consider it as a union of intervals belonging to $\Z^2 \cup \H^2$. We will take a union of these intervals whose length exceeds $\beta \delta/4$ and add it to $U_2$ if its length is less than $\beta\delta/2$, add it to $V_2$ if its length exceeds $\beta\delta/2$ and the last interval added was from $\Z^2$ and if its length exceeds $\beta\delta/2$ but the last interval added was from $\H^2$ then we simply remove the last part of this interval from the union so that the total union has length less than $\beta\delta/2$, and add it to $U_2$. 

Once the algorithm is finished, as before we end up with a partition $\U_\delta$ of $I$. After removing any intervals which do not intersect $\Lambda$ we obtain $\U_\delta'$ which satisfies \ref{Ub}-\ref{Ud}. In particular, the minimum length of any set in $\U_\delta'$ is $\frac{\delta}{2} (\beta/4)^{n_1}$.  Note that regarding \ref{Ud}, for any $U \in U_\delta'$ and $Z \in \Z^i$ ($1 \le i \le n_1$) where $U \cap Z \neq \emptyset$ and $Z$ is not a subset of $U$ we have
\begin{equation}
\frac{m(U \cap Z)}{m(Z)} \le \frac{\tilde{\beta}}{2(n_1+1)(1+C_d)^2}
\label{proportion}
\end{equation}
i.e. we have gained an extra factor of $1/2$.
However $\U_\delta'$ does not necessarily satisfy \ref{Ua} (we have upper and lower bounds on the diameter and since each $U \in \U_\delta'$ intersects $\Lambda$ then $\U_\delta$ satisfies the assumption involving $T$ in \ref{Ua} but not necessarily the assumption involving $t$ since for instance an interval in $\U_\delta$ may only intersect $\Lambda$ close to its boundary). We obtain a new family $\U_\delta''$ from $U_\delta'$ which additionally satisfies this property. For each $U \in \U_\delta'$ choose the leftmost point $x_U \in \Lambda\cap \U_\delta$ and fix an interval $[x_U-a_U,x_U+b_U]$ where 
$$ (\beta/4)^{n_1}\frac{\delta}{3}\le a_U,b_U \le (\beta/4)^{n_1}\frac{2\delta}{3}$$
and such that $[x_U-a_U,x_U+b_U]$ is equal to a union of $n_1+L_\delta$ cylinders. Let 
$$W_\delta=\{[x_U-a_U,x_U+b_U]\}.$$
 For any two intervals in $W_\delta$ whose interiors intersect we take their union, giving a new collection of closed intervals $W_\delta''$ with pairwise disjoint interiors. Finally we enlarge each interval in $W_\delta''$ so that their interiors remain pairwise disjoint, each interval intersects at most two intervals from $\U_\delta'$, each interval in $\U_\delta'$ is covered and so that any new interval formed is still given by a union of $n_1+L_\delta$ cylinders. We call this new collection $\U_\delta''$. \ref{Ua}-\ref{Uc} holds for $\U_\delta''$ by construction. \ref{Ud} holds since any $U \in \U_\delta''$ can contain at most two intervals from $U_\delta'$ and by \eqref{proportion}. \ref{Ue} holds since for each $U \in\U_\delta''$, $U \cap \Lambda$ is given by a finite or countable collection of $n_1+L_\delta$ cylinders.

\subsection{Proof of Proposition \ref{prop:u'}}

We now assume $f$ has at least 2 full branches. Choose $a, b \in \Sigma_1$ and $n_3 \in \N$ to satisfy Proposition \ref{prop:u'}(a) and define 
$$\V_\delta'=\left\{\Pi([wab^{n_3}]) \; : \; w \in \{a,b\}^* \; \textnormal{s.t.} \;\sup_x|Df_w^{-1}(x)|\le \delta < \sup_x|Df_{w^-}^{-1}(x)|\right\},$$
where $w^-$ denotes the word $w$ with the last digit removed and the suprema are taken over $\Lambda$ (since $f_a$ and $f_b$ are full branches). By definition, any two sets in $\V_\delta'$ can may intersect at most at one point, thus by removing at most half of the sets in $\V_\delta'$ we obtain a collection $\V_\delta'' \subset \V_\delta'$ of pairwise disjoint sets. Proposition \ref{prop:u'}(c) must be satisfied for $V_\delta''$ since the dimension of the repeller of $\{f_a,f_b\}$ is positive and each $V \in \V_\delta''$ has diameter $C_{bd}^{-1}\delta\le |V| \le C_{bd}^{-1}\delta$  by \eqref{bd}. To see Proposition \ref{prop:u'}(d) for $\V_\delta''$, note that there exists $t'>0$ and $x \in \Pi[ab^{n_3}]$ such that $B(x,t')\cap \Lambda \subset \Pi([ab^{n_3}])$. By \eqref{bd} given any $\Pi([wab^{n_3}]) \in \V_\delta''$, $B(f_w^{-1}(x), C_{bd} t'\delta) \cap \Lambda \subset \Pi([wab^{n_3}])$. The other part is similar.  Proposition \ref{prop:u'}(e) holds for $\V_\delta''$ provided $\delta_0$ is taken sufficiently small since we are only considering the finite branched system $\{f_a,f_b\}$. Proposition \ref{prop:u'}(f) holds for $\V_\delta''$ by uniform hyperbolicity. Proposition \ref{prop:u'}(g) holds for $\V_\delta'$ (therefore also for $\V_\delta''$) since if $\Pi([w_1ab^{n_3}]), \Pi([w_2ab^{n_3}]) \in \V_\delta'$ are distinct and $w_1ab^{n_3}$ is a subword of $w_2ab^{n_3}$ this contradicts that $\Pi([w_1ab^{n_3}]) \in \V_\delta'$ (the bounds on the derivative $Df_{w_1}^{-1}$ cannot hold). However, Proposition \ref{prop:u'}(b) may not hold for $\V_\delta''$ since  $\sup_x |D f_{w_0 \ldots w_n}^{-1}(x)|  \le \delta < \sup_x|Df_{w_0 \ldots w_{n-1}}^{-1}(x)|$ does not exclude the possibility that $\sup_x |D f_{w_1\ldots w_n}^{-1}(x)|  \le \delta < \sup_x|Df_{w_1 \ldots w_{n-1}}^{-1}(x)|$. However, if we let $k \in \N$ be sufficiently large that 
$$\frac{\inf_x\min\{Df_a^{-1}(x),Df_b^{-1}(x)\}}{\left(\sup_x\min\{Df_a^{-1}(x),Df_b^{-1}(x)\}\right)^k} <1$$
noting that this is possible since $f$ is uniformly expanding, then whenever $\sup_x |D f_{w_0 \ldots w_n}^{-1}(x)|  \le \delta < \sup_x|Df_{w_0 \ldots w_{n-1}}^{-1}(x)|$ we have $\sup_x|Df_{w_{k} \ldots w_{n}}^{-1}(x)|>\delta$. In particular $w_\ell \ldots w_n ab^{n_3} \notin \V_\delta''$ for any $\ell \ge k$. Therefore, we can find $\V_\delta \subset \V_\delta''$, where $\# \V_\delta \ge \frac{1}{k}\# \V_\delta''$, so that Proposition \ref{prop:u'}(b) holds for $\V_\delta$. Moreover, the remaining parts of Proposition \ref{prop:u'} also hold for $\V_\delta$. 



\begin{thebibliography}{XXX}


\bibitem[AL]{AbaLam13} M.\ Abadi, R.\ Lambert, 
\emph{The distribution of the short-return function,}
Nonlinearity \textbf{26} (2013) 1143--1162.

\bibitem[Al1]{aldous-graph} D.\ Aldous, \emph{An introduction to covering problems for random walks on graphs},
J. Theoret. Probab., (1) \textbf{2}, (1989) 87--89.

\bibitem[Al2]{aldous} D.\ Aldous, \emph{Covering a compact space by fixed-radius or growing random balls},
ALEA Lat. Am. J. Probab. Math. Stat. \textbf{19} (2022) 755--767.

\bibitem[Al3]{covertree} D.\ Aldous \emph{Random walk covering of some special trees}, J. Math. Anal. Appl. \textbf{157} (1991) 271--283.

\bibitem[ARS]{genomics} S. Athreya, R. Rahul and A, Sarkar. \emph{On the coverage of space by random sets.} Advances in Applied Probability \textbf{36.1} (2004) 1--18.

\bibitem[BB]{wireless} F. Baccelli and B. Błaszczyszyn. \emph{Stochastic Geometry and Wireless Networks: Volume I Theory.} Foundations and Trends in Networking \textbf{3.3–4} (2010) 249--449.

\bibitem[BJK]{bjk} B. B\'ar\'any, N.\ Jurga, I. \ Kolossv\'ary, 
\emph{On the convergence rate of the chaos game}, Int. Math. Res. Not. IMRN (2023), no. 5, 4456--4500. 

\bibitem[Ba]{BarnsleyBook} M.\ Barnsley, \emph{Fractals everywhere}, Academic Press Professional, Inc., 1988.

\bibitem[Br]{Bra83} R.C.\ Bradley,
\emph{ On the $\psi$-mixing condition for stationary random sequences,}
 Trans. Amer. Math. Soc. \textbf{276} (1983) 55--66.

\bibitem[BDT]{BruDemTod18} H.\ Bruin, M.F.\ Demers, M.\ Todd, 
\emph{Hitting and escaping statistics: mixing, targets and holes,}
 Adv. Math. \textbf{328} (2018) 1263--1298.
 
 \bibitem[D]{Dem05} M.F.\ Demers, \emph{Markov extensions for dynamical systems with holes: an application to expanding maps of the interval,} Israel J. Math. \textbf{146} (2005) 189--221. 

\bibitem[DLP]{dlp} J. Ding, J.R. Lee, and Y. Peres. \emph{Cover times, blanket times, and majorizing measures.} Ann. of Math. \textbf{175.3} (2012) 1409--1471.

\bibitem[DPRZ]{coverdisk} A. Dembo, Y. Peres, J. Rosen, and O. Zeitouni. \emph{Cover times for Brownian motion and random walks in two dimensions.} Ann. of Math. (2) \textbf{160} (2004) 433--464.


\bibitem[F]{flatto} L. Flatto, D. J. Newman,
\emph{Random coverings}, Acta Mathematica, \textbf{138} (1977) 241--264.

\bibitem[FFK]{ffk} K. J. Falconer, J. M. Fraser, A. K\"aenm\"aki, 
\emph{Minkowski dimension for measures, } Proc. Amer. Math. Soc. \textbf{151} (2023) 779--794.


\bibitem[FFTV]{FreFreTodVai16} A.C.M.\ Freitas, J.M.\ Freitas, M.\ Todd, S.\ Vaienti, 
\emph{Rare events for the Manneville-Pomeau map,} Stochastic Process. Appl. \textbf{126} (2016) 3463--3479. 



\bibitem[G1]{Gan96} N.\ Gantert, \emph{Large deviations for a heavy-tailed mixing sequence,} Preprint TU Berlin, 1996.

\bibitem[G2]{Gan00} N.\ Gantert, \emph{A note on logarithmic tail asymptotics and mixing,} Statist. Probab. Lett. \textbf{49} (2000) 113--118. 

\bibitem[J]{janson} S. Janson. \emph{Random coverings in several dimensions.} Acta Mathematica \textbf{156} (1986: 83--118.

\bibitem[JM]{JurMor20} N.\ Jurga, I.D.\ Morris, 
\emph{How long is the Chaos Game?} Bull. Lond. Math. Soc. \textbf{53} (2021) 1749--1765.

\bibitem[KL1]{kl} G. Keller, C. Liverani,
  \emph{Stability of the spectrum for transfer operators,}
  Ann. Scuola Norm. Sup. Pisa Cl. Sci. \textbf{28} (1999) 141--152

\bibitem[KL2]{kl2}  G. Keller, C. Liverani,
  \emph{Rare events, escape rates and quasistationarity: some exact
   formulae,}
J. Stat. Phys. \textbf{135} (2009) 519--534

\bibitem[L]{lovasz} L. Lovász. \emph{Random walks on graphs: a survey,}  Comb. Paul Erd\"os is Eighty \textbf{2} 1--46 (1993).

\bibitem[LP]{peres} D. Levin, Y. Peres, \emph{Markov chains and mixing times}, Second Edition, American Mathematical Society, Providence, RI, 2017.

\bibitem[P]{pesin} Y. B. Pesin. \emph{Dimension theory in dynamical systems: contemporary views and applications}. University of Chicago Press, 2008.

\bibitem[LFF$^+$]{EVbook}
V.\ Lucarini, D.\ Faranda, A.C.M.\ Freitas, J.M.\  Freitas, M.\ Holland, T.\ Kuna, M.\ Nicol, M.\ Todd, S.\
  Vaienti, \emph{Extremes and recurrence in dynamical systems}, Pure and
  Applied Mathematics: A Wiley Series of Texts, Monographs and Tracts, Wiley,
  Hoboken, NJ, 2016.

\bibitem[M]{matthews} P. Matthews. \emph{Covering problems for brownian motion on spheres}, Ann. Probab.  \textbf{16} (1988) 189--199.

\bibitem[P]{penrose} M. Penrose. \emph{Random Euclidean coverage from within}, Probab. Theory Related Fields \textbf{185} (2023) 747--814.






\bibitem[R]{rychlik} M.\ Rychlik,
\emph{Bounded variation and invariant measures,}
 Studia Math. \textbf{76} (1983) 69--80.


\bibitem[S1]{Sar02} O.\ Sarig, \emph{Subexponential decay of correlations,} Invent. Math. \textbf{150} (2002) 629--653.

\bibitem[S2]{Sar03} O.\ Sarig, \emph{Existence of Gibbs measures for countable Markov shifts,} Proc. Amer. Math. Soc. \textbf{131} (2003) 1751--1758

\bibitem[Sau]{Sau09}
B.\ Saussol, \emph{An introduction to quantitative {P}oincar\'e recurrence in dynamical systems,}
 Rev. Math. Phys. \textbf{21} (2009) 949--979.

\end{thebibliography}
\end{document}